\documentclass[11pt, a4paper,twoside,reqno]{amsart}

\usepackage[ a4paper, left=1.75cm, right=1.75cm, top=1.5cm, bottom=1.5cm, includeheadfoot ]{geometry} 
\makeatletter \def\@settitle{%
\begin{center} \baselineskip14\p@\relax \normalfont\LARGE\scshape\bfseries \@title \end{center} } \def\@setauthors{ \begingroup \def\thanks{\protect\thanks@warning}
\trivlist \centering\small \@topsep30\p@\relax \advance\@topsep by -\baselineskip \item\relax \author@andify\authors \def\\{\protect\linebreak}
\authors \ifx\@empty\contribs \else ,\penalty-3\space\@setcontribs \@closetoccontribs \fi \endtrivlist \endgroup } \makeatother 

\makeatletter \def\subsection{\@startsection{subsection}{2}%
\z@ {.5\linespacing\@plus.7\linespacing}%
{.5\linespacing}%
{\normalfont\large\bfseries}} \def\subsubsection{\@startsection{subsubsection}{3}%
\z@ {.5\linespacing\@plus.7\linespacing}%
{.5\linespacing}%
{\normalfont\itshape}} \makeatother 
\usepackage[usenames,dvipsnames]{xcolor} \definecolor{darkblue}{rgb}{0.0,0.0,0.45} 
\usepackage{amsmath} 
\usepackage{amssymb} 
\usepackage{amsfonts} 
\usepackage{amsthm} 
\usepackage{mathtools} 
\usepackage{mathrsfs} 
\usepackage{dsfont} 
\usepackage{graphicx} 
\usepackage{multirow} 
\usepackage{makecell} 
\usepackage{fancyhdr} \usepackage[amssymb,thickqspace]{SIunits} 
\usepackage[dvipsnames]{xcolor} \usepackage{hyperref} \definecolor{LinkBlue}{RGB}{25, 75, 135} \definecolor{CiteRed}{RGB}{155, 35, 45} \definecolor{UrlGreen}{RGB}{20, 105, 85} \hypersetup{ colorlinks=true, linkcolor=LinkBlue, citecolor=CiteRed, urlcolor=UrlGreen, filecolor=LinkBlue, menucolor=LinkBlue, anchorcolor=LinkBlue, pdfauthor={Mohammad Boveiri and Peyman Mohajerin Esfahani}, pdftitle={Discrete-Time Adaptive Control in High Dimensions: Near Dimension-Free Performance via Mirror Descent}, pdfsubject={Adaptive control of high-dimensional discrete-time nonlinear systems}, pdfkeywords={adaptive control, mirror descent, high-dimensional systems, sparse parameters, low-rank matrices, regret bounds}, pdfcreator={LaTeX}, pdfproducer={pdfTeX}, bookmarks=true, bookmarksopen=true, bookmarksnumbered=true, pdfstartview=FitH, unicode=true, plainpages=false, raiselinks=true }
\allowdisplaybreaks 
\date{\today}

\usepackage{changepage} 
\usepackage{textcomp} 
\usepackage{verbatim} 
\usepackage{empheq} 

\usepackage[most]{tcolorbox} 
\usepackage[ backend=biber, style=ieee, citestyle=numeric-comp, sorting=none, giveninits=true ]{biblatex}
\usepackage{xcolor}
\usepackage{subcaption}
    \theoremstyle{plain}
    \newtheorem{theorem}{Theorem}
    \newtheorem*{theorem*}{Theorem}
    \newtheorem*{definition*}{Definition}
    \newtheorem{assumption}{Assumption}
    \newtheorem{corollary}[theorem]{Corollary}
   \newtheorem{lemma}[theorem]{Lemma}
    \newtheorem{remark}{Remark}
    \newtheorem*{example*}{Example}

    \newtheorem*{problem*}{Problem}

\DeclareMathOperator{\Tr}{Tr}

\def\BibTeX{{\rm B\kern-.05em{\sc i\kern-.025em b}\kern-.08em
    T\kern-.1667em\lower.7ex\hbox{E}\kern-.125emX}}
\begin{document}
\allowdisplaybreaks
\pagestyle{empty}

\title{ Discrete-Time  Adaptive Control in High Dimensions:\\ Near Dimension-Free  Performance via Mirror Descent}

\author[M. Boveiri]{Mohammad Boveiri$^{1}$}
\author[P. {Mohajerin Esfahani}]{Peyman {Mohajerin Esfahani$^{1,2}$}
\\
\\
$^1$Delft University of Technology, The Netherlands\\
$^2$University of Toronto, Canada}
\thanks{This work was partially supported by the ERC grant TRUST-949796 and the NSERC Grant RGPIN-2025-06544.}
\maketitle
\begin{abstract}
Motivated by the use of modern high-capacity models in real-time control problems, this paper studies the adaptive control of high-dimensional discrete-time nonlinear systems with an unknown matrix-valued parameter. We focus on regimes where the number of unknown parameter entries is large, but the parameter matrix possesses exploitable structure, such as entrywise sparsity, group sparsity, low rank, and row-stochastic or density-matrix structure. To quantify transient performance, we consider a regret criterion relative to a nominal controller with full knowledge of the true parameter. We show that standard Euclidean update schemes, including recursive least squares and gradient descent, are ill-suited to this setting: their transient performance deteriorates as the dimension increases, even when the true parameter has low intrinsic complexity. To address this limitation, we propose a novel class of mirror-descent-type adaptive laws equipped with a non-Euclidean Polyak-type step size that exploit the geometry induced by the parameter structure. For the proposed update laws, we establish asymptotic state convergence and derive regret bounds with at most logarithmic dependence on the ambient dimension. Numerical experiments demonstrate the effectiveness of the proposed schemes. 
\end{abstract}

\section{Introduction}
\label{sec:introduction}

\subsection{Problem description}

In this paper, we consider the adaptive control of uncertain discrete-time nonlinear systems of the form \begin{equation}\label{eq.dynamics} X_{t+1} = \mathcal{A}(X_t) + \mathcal{B}\bigl(\Theta\Psi(X_t,t)+U_t\bigr), \end{equation} where $X_t \in \mathbb{R}^{n_1 \times n_2}$ is the system state, $U_t \in \mathbb{R}^{m \times n_2}$ is the control input, and $\Theta \in \Omega \subseteq \mathbb{R}^{m \times k}$ is an unknown parameter matrix. The known mapping $\Psi:\mathbb{R}^{n_1 \times n_2}\times\mathbb{Z}_{+} \to\mathbb{R}^{k\times n_2}$ is a possibly time-varying feature map, or regressor. The known linear operators $\mathcal{A}:\mathbb{R}^{n_1 \times n_2} \to\mathbb{R}^{n_1 \times n_2}$ and $\mathcal{B}:\mathbb{R}^{m \times n_2} \to\mathbb{R}^{n_1 \times n_2}$ describe the nominal dynamics and the control channel, respectively. Since the uncertain term $\Theta\Psi(X_t,t)$ enters the dynamics through the same channel as the control input, the system has a matched uncertainty structure \cite{Krstic}. Throughout this study, we allow the regressor to be unbounded, provided that it grows at most linearly with the state, as formalized below. \begin{assumption}[Linear growth] \label{ass:linear-growth-regressor} The regressor $\Psi$ satisfies \begin{equation}\label{eq:linear-growth-regressor} \|\Psi(X,t)\| \leq c_1+c_2\|X\|, \qquad \forall (X,t)\in \mathbb{R}^{n_1\times n_2}\times\mathbb{Z}_{+}, \end{equation} for some constants $c_1,c_2\in \mathbb{R}_{+}$, which need not be known. \end{assumption}

The control task is to track the desired state trajectory $\{X^{\rm d}_t\}_{t \in \mathbb{Z}_{+}}$ generated by the reference model
\begin{equation}\label{eq.dynamcis.ref}
    X^{\rm d}_{t+1} = (\mathcal{A} - \mathcal{B}\mathcal{K})(X^{\rm d}_t) + \mathcal{B} (U^{\rm d}_t),
\end{equation}
where $U_t^{\mathrm d} \in \mathbb{R}^{m \times n_2}$ is a bounded reference input and $\mathcal{K}: \mathbb{R}^{n_1 \times n_2} \to \mathbb{R}^{m \times n_2}$ is a linear feedback operator chosen such that the closed-loop operator $\mathcal{A} - \mathcal{B}\mathcal{K}$ is Schur stable. We are particularly interested in high-dimensional settings where the number of unknown parameters, $d=mk$, is large. 

Our objective is to develop adaptive control schemes with performance guarantees that are dimension-free, up to logarithmic factors, by leveraging structural properties of $\Theta$, such as sparsity, low-rank structure, or stochastic matrix structure. Although the system in \eqref{eq.dynamics} can be rewritten in standard vector form by stacking the columns of $X_t$, we retain the matrix representation in order to exploit structural properties of the uncertainty term $\Theta \Psi(X_t,t)$, such as low-rank structure, which are more naturally expressed in matrix form.

\textbf{Control law.} Given an estimate $\widehat{\Theta}_t$ of the unknown matrix $\Theta$ at time $t$, we apply the certainty-equivalent control law
\begin{equation}\label{eq:control.law}
    U_t(X_t, \widehat{\Theta}_t) = -\mathcal{K}(X_t) - \widehat{\Theta}_t \Psi(X_t,t)+U^{\rm d}_t.
\end{equation}
  The term `certainty-equivalent' indicates that the above controller is designed by treating the estimate $\widehat{\Theta}_t$ as the true parameter $\Theta$; in the ideal case where $\widehat{\Theta}_t = \Theta$, the control law \eqref{eq:control.law} ensures the cancellation of the unknown nonlinearity and convergence of $X_t$ to $X^{\rm d}_t$.  In the adaptive setting, the estimate $\widehat{\Theta}_t$ is updated online based on the observed state evolution.
  
  We couple the controller~\eqref{eq:control.law} with a family of non-Euclidean parameter update laws. These updates are driven by the instantaneous performance loss induced by the parameter-estimation error $\Theta-\widehat{\Theta}_t$ in the closed-loop system. Specifically, after applying the control input at time $t$, we define
\begin{equation}\label{eq:cost}
    J_t(\widehat{\Theta}_t)
    :=
    \frac{1}{2}
    \left\|
        \mathcal{B}\big((\Theta-\widehat{\Theta}_t)\Psi(X_t,t)\big)
    \right\|_\mathrm{F}^2
    =
    \frac{1}{2}\|\widetilde{X}_{t+1}\|_\mathrm{F}^2 ,
\end{equation}
where 
$\widetilde{X}_{t+1}
    :=
    X_{t+1}
    -
    (\mathcal{A}-\mathcal{B}\mathcal{K})(X_t)
    -
    \mathcal{B}(U^{\rm d}_t)$.
The gradient of $J_t$ with respect to $\widehat{\Theta}_t$ is 
\begin{equation}\label{eq:cost.gradient}
    \nabla J_t(\widehat{\Theta}_t)
    =
    -\mathcal{B}^{\top}\big(\widetilde{X}_{t+1}\big)
    \Psi(X_t,t)^\top ,
\end{equation}
where $\mathcal{B}^{\top}$ denotes the adjoint of the linear operator $\mathcal{B}$; see Section~\ref{sec:main.prel}. Observe that, although $J_t$ is associated with the control action applied at time $t$, both $J_t$ and $\nabla J_t(\widehat{\Theta}_t)$ become available only at time $t+1$, once the residual $\widetilde{X}_{t+1}$ can be computed from the observed state $X_{t+1}$.

\subsection{Contributions}
\textbf{Parameter update law.}
Let $\mathbb{R}^{m\times k}$ be equipped with a norm $\|\cdot\|$, and let
$\|\cdot\|_*$ denote the associated dual norm. Moreover, let
$f:\mathbb{R}^{m\times k}\to\mathbb R\cup\{+\infty\}$ be a $\mu$-strongly convex function on
$\Omega\subset\mathbb{R}^{m\times k}$ with respect to $\|\cdot\|$, and let
$f^*$ denote its Fenchel conjugate. The estimate
$\widehat{\Theta}_t$ in the control law \eqref{eq:control.law} is updated
according to the following mirror-descent-type recursion:
\begin{tcolorbox}[ colback=blue!6, colframe=blue!60!black, boxrule=0.6pt, arc=1mm, left=6pt, right=6pt, top=4pt, bottom=4pt ] \begin{subequations}\label{eq:update.law} \begin{equation}\label{eq:update.law.main} \begin{aligned} Z_{t+1} &= Z_t-\eta_{t+1}\nabla J_t(\widehat{\Theta}_t), \qquad Z_0\in\partial f(\widehat{\Theta}_0),\\ \widehat{\Theta}_{t+1} &=\nabla f^*(Z_{t+1}), \end{aligned} \end{equation} where the step size is selected as \begin{equation}\label{eq:step-size} \eta_{t+1} = \frac{2\mu\,J_t(\widehat{\Theta}_t)} {\epsilon_{t+1} +\lVert\nabla J_t(\widehat{\Theta}_t)\rVert_*^2} \end{equation} for some nonnegative and bounded sequence $\{\epsilon_{t+1}\}_{t\in\mathbb{Z}_+}$. \end{subequations} \end{tcolorbox}
\noindent Interestingly, \eqref{eq:step-size} may be viewed as a non-Euclidean and online 
generalization of the Polyak step size; see Remark~\ref{remark.stepsize}. Within this framework, we establish the following results.

\textbf{(i) Stability and state convergence.} We establish boundedness and asymptotic convergence properties of the closed-loop system under the parameter update law \eqref{eq:update.law}. Specifically, we show that the system state $X_t$ and parameter estimate $\widehat{\Theta}_t$ remain bounded and that $X_t$ converges asymptotically to the desired trajectory $X_t^{\rm d}$ (Theorem~\ref{Theorem.stability}). We emphasize that the step-size choice \eqref{eq:step-size} is essential for guaranteeing stability and state convergence when the map $\Psi$ is not necessarily bounded and may grow with the state according to Assumption~\ref{ass:linear-growth-regressor}.

\textbf{(ii) Performance.}
We quantify transient performance via the cumulative loss (regret) defined by
\begin{equation}\label{eq:regret.intro}
        \operatorname{Reg}(T):=\sum_{t=0}^{T-1} J_t(\widehat{\Theta}_t),\quad\quad T>0.
\end{equation}
Observe that a controller of the form \eqref{eq:control.law} with full knowledge of the true parameter $\Theta$, namely $U_t(X_t,\Theta)$,
incurs zero regret, since $J_t(\Theta)=0$ for all $t$.
Thus, \(\operatorname{Reg}(T)\) characterizes the cumulative performance loss over the horizon \(T\) caused by estimating $\Theta$ through \eqref{eq:update.law}, rather than having access to it a priori. Given this setup, we show that the regret admits the bound
\begin{equation}\label{eq:intro.regret.bound}
     \operatorname{Reg}(T)
     \leq
     \sqrt{
     \frac{D_f(\Theta,\widehat{\Theta}_0;Z_0)}{2\mu}
    \sum_{t=0}^{T-1} \left(\epsilon_{t+1}+
        \| \nabla J_t(\widehat{\Theta}_t)\|_{*}^{2}
     \right)
     },
\end{equation}
where  $D_f$ denotes the generalized Bregman divergence generated by the mirror map $f$ (Theorem~\ref{Theorem.regret}). Since boundedness of $X_t$ and $\widehat{\Theta}_t$ has already been established, $\|\nabla J_t(\widehat{\Theta}_t)\|_{*}$ is bounded. Hence, \eqref{eq:intro.regret.bound} implies that $\operatorname{Reg}(T)$ scales at most as $\sqrt{T}$ with the horizon.

\textbf{(iii) Dimension-free regret via geometry-aware updates.} 
The regret bound \eqref{eq:intro.regret.bound} depends on
$D_f$ and $\|\cdot\|_{*}$, and thus on the geometry
induced by the mirror map $f$. For example, with the Euclidean mirror map
$f(\Theta)=\frac{1}{2}\|\Theta\|_\mathrm{F}^2$, the update law
\eqref{eq:update.law.main} reduces to a gradient-descent-type update. In this case, the
dual norm is the Frobenius norm, and, assuming $\widehat{\Theta}_0=0$, 
\begin{equation*}
  D_f(\Theta,\widehat{\Theta}_0;Z_0)=\frac{1}{2}\|\Theta\|_\mathrm{F}^2, 
  \quad
  \|\nabla J_t(\widehat{\Theta}_t)\|_{*}
  =\|\nabla J_t(\widehat{\Theta}_t)\|_\mathrm{F}.
\end{equation*}
Hence, even if $\Theta$ is $s$-sparse with $s\ll d$, the regret is $\mathcal{O}(\sqrt{d})$, since $\|\nabla J_t(\widehat{\Theta}_t)\|_\mathrm{F}=\mathcal{O}(\sqrt{d})$.    
This dimensional dependence of Euclidean-based schemes is not an artifact of conservative analysis; it is also observed empirically (see Section~\ref{sec:moti}).

In contrast, by selecting the mirror map according to the structure of $\Theta$, we show that regret bounds with at most logarithmic dependence on the dimension can be achieved. 
We consider two classes of structured parameter sets:

\begin{itemize}
\item \textbf{Sparsity structures.} We consider several notions of sparsity for $\Theta$, including (a) entrywise sparsity, where at most $s$ entries of $\Theta$ are nonzero; (b) group sparsity, where the nonzero entries of $\Theta$ are structured into groups, such as in row-wise or column-wise sparsity; and (c) low-rank structure, where $\operatorname{rank}(\Theta) \leq s$; see Section~\ref{sec:sparse}.

    \item \textbf{Simplex-type sets.} We also consider the case where the
    unknown parameter $\Theta$ belongs to a simplex-type constraint set.
    Specifically, we study two representative examples: the standard
    probability simplex and the spectraplex (i.e., the set of density
    matrices); see Section~\ref{sec:simplex}.
\end{itemize}  
For each case above, we identify a suitable mirror map, derive the corresponding explicit update law, and establish regret bounds with at most logarithmic dependence on the dimension.

\subsection{Motivation and related literature}
Historically, one of the primary objectives of adaptive control theory has been the development of universal controllers capable of managing broad classes of dynamical systems across diverse operating scenarios \cite{astromhist,ANNASWAMY202118}. This goal naturally necessitates high-dimensional modeling to capture complex system dynamics, exogenous disturbances, and potential faults. Traditionally, however, overparameterized modeling has been viewed as problematic (see, e.g., \cite{Krstic1992AdaptiveNonlinear}) primarily due to the severe degradation in transient performance observed in existing schemes as the dimensionality increases (see Section~\ref{sec:moti}). This tension between expressive, high-capacity modeling and reliable transient behavior in adaptive control design has been a central obstacle to integrating modern high-capacity models into real-time control systems.

In this paper, we show that such performance degradation is not an inherent consequence of using high-dimensional models. Rather, it arises largely from the use of Euclidean parameter update laws, which fail to preserve or exploit structural properties commonly present in high-dimensional models, such as sparsity. By incorporating such structural properties directly into the geometry of the parameter update law, we obtain transient performance guarantees that depend on the intrinsic complexity of the unknown parameter rather than on the ambient
dimension.
We believe that, by addressing this fundamental limitation, this paper paves the way for the use of high-capacity models in adaptive control without sacrificing transient performance.

To the best of the authors' knowledge, this work provides the first dimension-independent performance guarantees for discrete-time adaptive systems in high-dimensional settings.

\textbf{Mirror descent}. Mirror-descent algorithms, introduced in
\cite{NemirovskiYudin1983}, have been widely used in optimization and
machine learning problems, including sparse signal recovery
\cite{juditsky2023sparse}, $\ell_1$- and nuclear-norm minimization
\cite{nesterov2013first}, online optimization \cite{Jadbabaie,tac_md},
and saddle-point problems \cite{Nemirovski2004MirrorProx,nemirovski2009robust}. From a purely algorithmic perspective, several aspects of our proposed mirror update law, such as its Polyak-type step size \eqref{eq:step-size} and the allowance for a nondifferentiable mirror map, are nonstandard and specifically tailored to the adaptive control framework. These features  address challenges that are unique to the adaptive control setting, such as establishing stability of the interconnected plant--controller system, boundedness of the closed-loop signals, and asymptotic convergence of the tracking error. 

\textbf{Regret as a measure of transient performance.} Regret is a performance measure that evaluates an algorithm’s performance relative to a benchmark strategy. It has become a standard criterion for assessing the finite-time performance of algorithms in online decision-making problems, including online convex optimization \cite{Hazan2016}, multi-armed bandits \cite{LattimoreSzepesvari2020,BubeckC12}, and reinforcement learning and control \cite{HazanSingh2026}. In adaptive control, regret quantifies the performance loss incurred by a learning controller relative to a benchmark controller with full knowledge of the system dynamics. Existing regret analyses for adaptive control have primarily emphasized the dependence of the bound on the time horizon $T$, rather than on the ambient dimension of the unknown parameters. For instance, in linear quadratic regulator problems, sublinear regret bounds, often of order $\sqrt{T}$, have been established under various assumptions on data generation, stability, excitation, and stochastic noise processes \cite{pmlr-v19-abbasi-yadkori11a,NEURIPS2018_0ae3f79a,Fara_tac,florian,NEURIPS.adap}. The regret definition adopted in this work is similar in spirit to that of \cite{pmlr-v144-boffi21a}, measuring the cumulative performance gap between our adaptive controller and a certainty-equivalence benchmark with access to the true system parameter.

\textbf{Paper structure.} The remainder of this paper is organized as follows. Section~\ref{sec:moti} presents a scalar motivating example illustrating how the transient performance of standard Euclidean-based parameter update laws deteriorates as the ambient dimension grows. Section~\ref{sec:main.result} develops the proposed mirror descent adaptive control framework, proves the boundedness of the closed-loop signals, establishes state convergence, and derives a general regret bound for the proposed scheme. Section~\ref{sec:sparse} specializes the framework to parameter classes with specific sparsity structures, including entrywise sparsity, group sparsity, and low-rank structure, and derives the corresponding regret bounds. Section~\ref{sec:simplex} extends the approach to simplex-type constraint sets, such as the standard simplex and the spectraplex. Section~\ref{sec:stochastic} extends the proposed update law to stochastic systems subject to process noise. Section~\ref{sec:simu} presents numerical simulations that validate the theoretical results. Finally, Section~\ref{sec:conclusions} concludes the paper and outlines potential directions for future research.

\textbf{Notation.}
The symbols $\mathbb{R}$, $\mathbb{R}_{+}$, and $\mathbb{Z}_{+}$ denote the sets of real numbers, nonnegative real numbers, and nonnegative integers, respectively.
For a matrix $Q \in \mathbb{R}^{m\times k}$, $\|Q\|_{p}$ denotes the entrywise $\ell_p$-norm, i.e.,
$\|Q\|_{p} = \left(\sum_{j,i} |Q_{ij}|^p\right)^{1/p}$ for $p\in[1,\infty)$, while $\|Q\|_{\infty}=\max_{i,j}|Q_{ij}|$. Moreover,
$\|Q\|_{\mathrm{S}_p}$ denotes the Schatten $p$-norm, and $\|Q\|_{\mathrm{F}}$ denotes the Frobenius norm.
For a linear operator $\mathcal{A}$, $\|\mathcal{A}\|_{\mathrm{op}}$ denotes its operator norm.
$\mathbb{S}_{+}^{n}$ denotes the cone of $n\times n$ symmetric positive semidefinite matrices.
The Hadamard product is denoted by $\odot$.
For a function $f$, $\partial f$ denotes its subdifferential; if $f$ is differentiable, then $\partial f = \{\nabla f\}$. The cardinality of a set $\mathcal{S}$ is indicated by $|\mathcal{S}|$.

\section{Motivating Example: Euclidean-Based Schemes in High Dimensions}\label{sec:moti}
\begin{figure*}[t]
    \centering
    \hspace{-0.4cm}
    \subfloat[Dimension: $10$]{
        \includegraphics[width=0.24\textwidth]{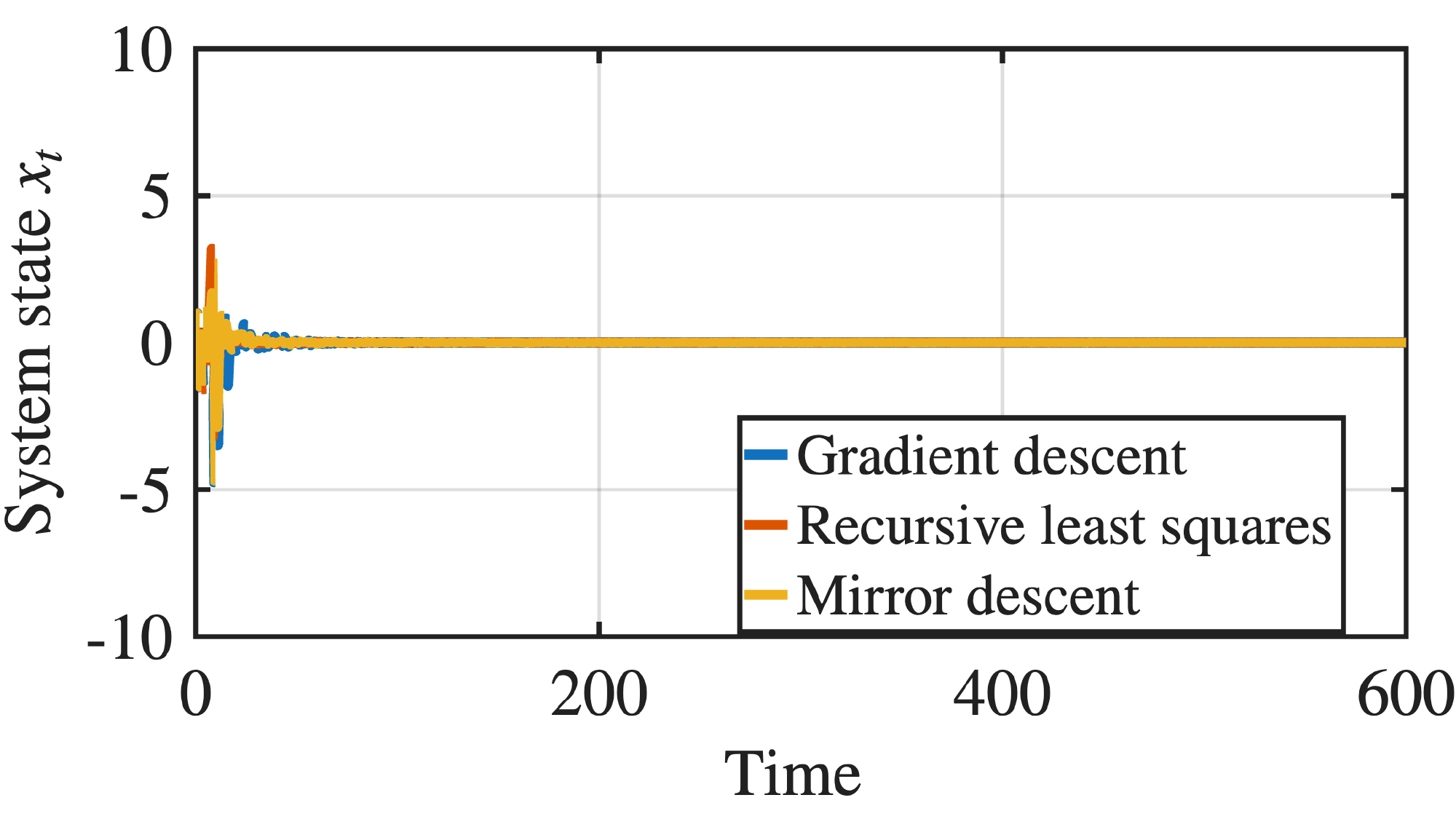}
        \label{fig:par_10}
    } \hspace{-0.4cm}
    \subfloat[Dimension: $50$]{
        \includegraphics[width=0.24\textwidth]{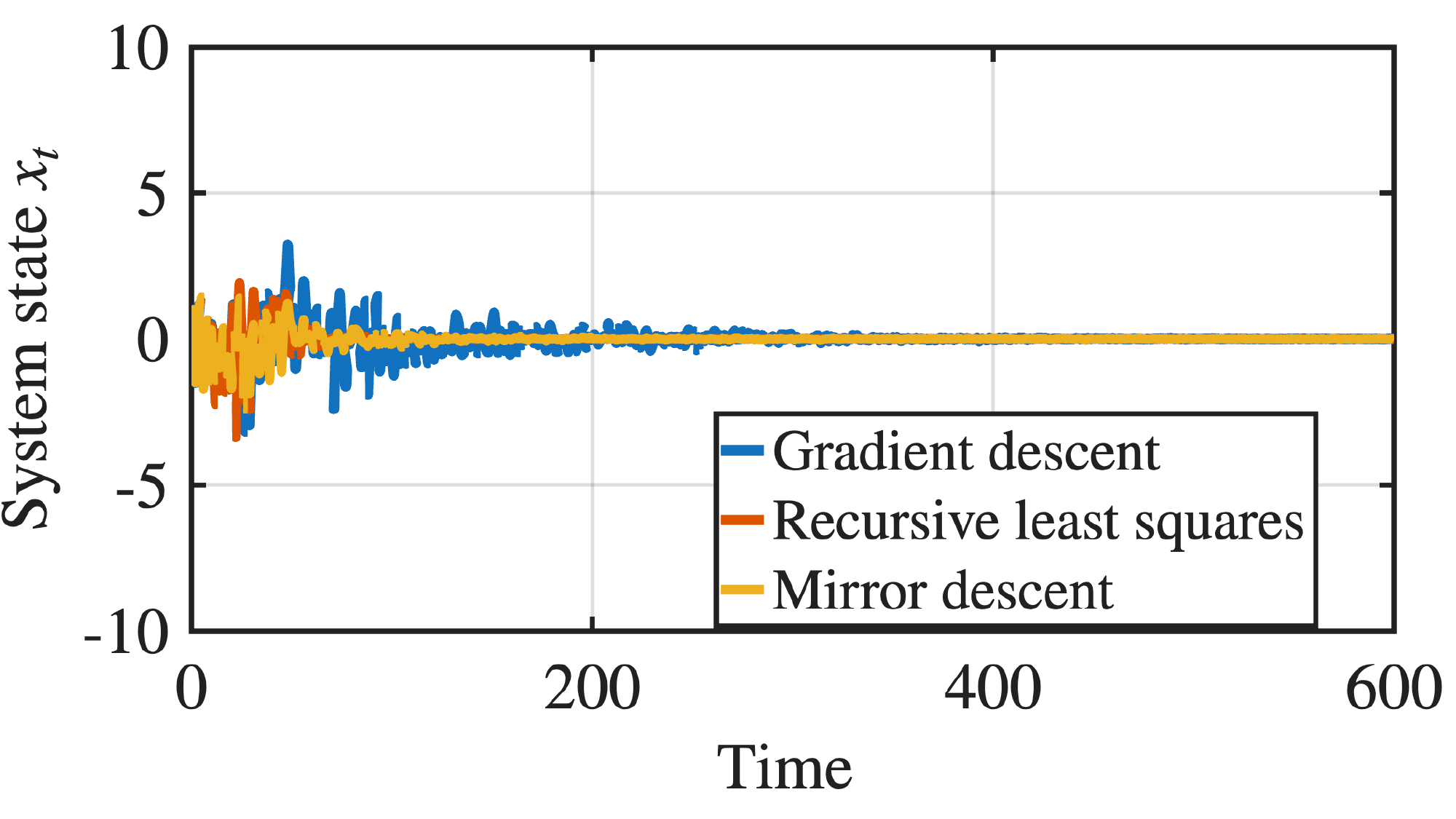}
        \label{fig:par_50}
    } \hspace{-0.4cm}
    \subfloat[Dimension: $500$]{
        \includegraphics[width=0.24\textwidth]{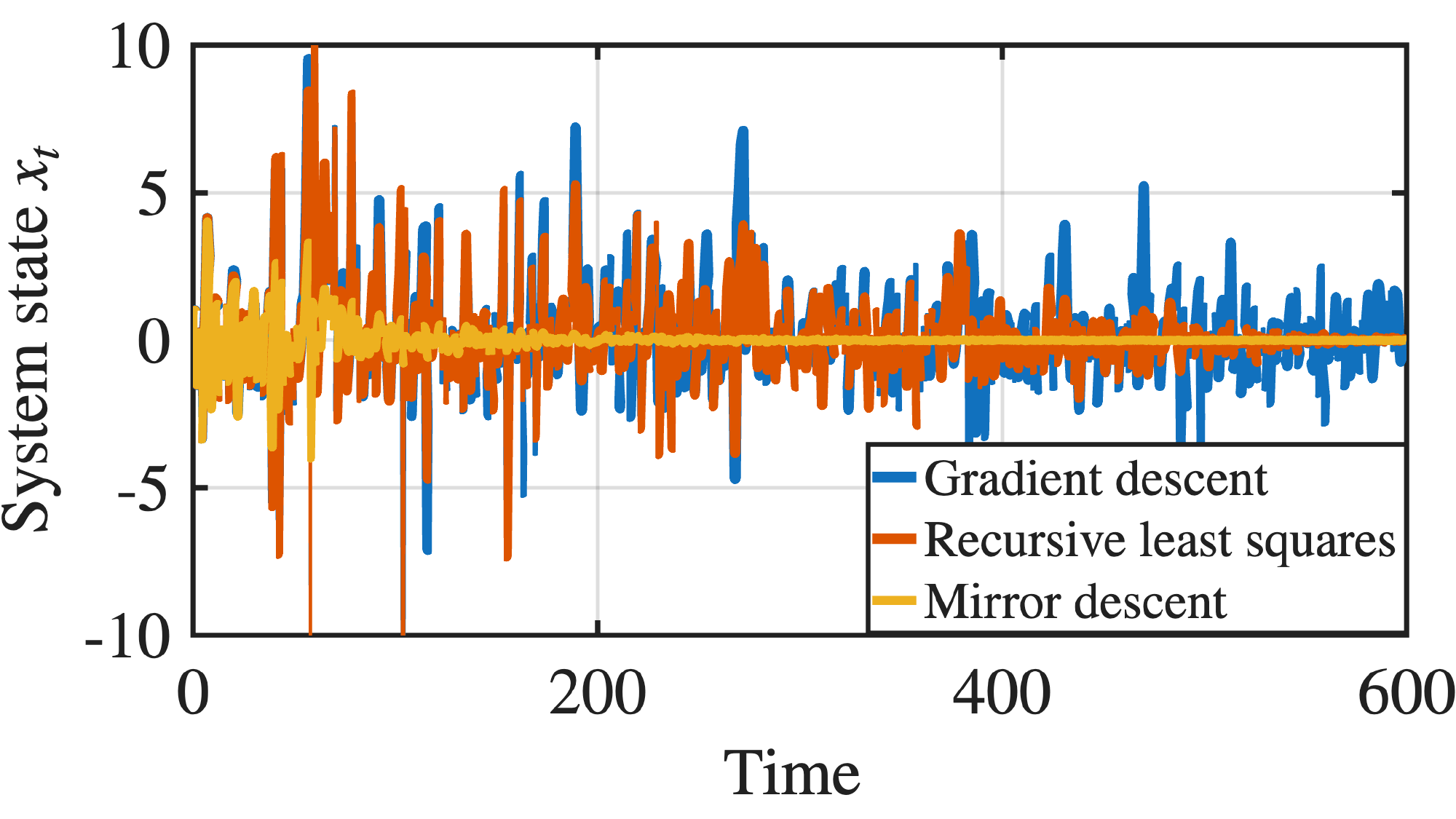}
        \label{fig:par_500}
    } \hspace{-0.4cm}
    \subfloat[Dimension: $2000$]{
        \includegraphics[width=0.24\textwidth]{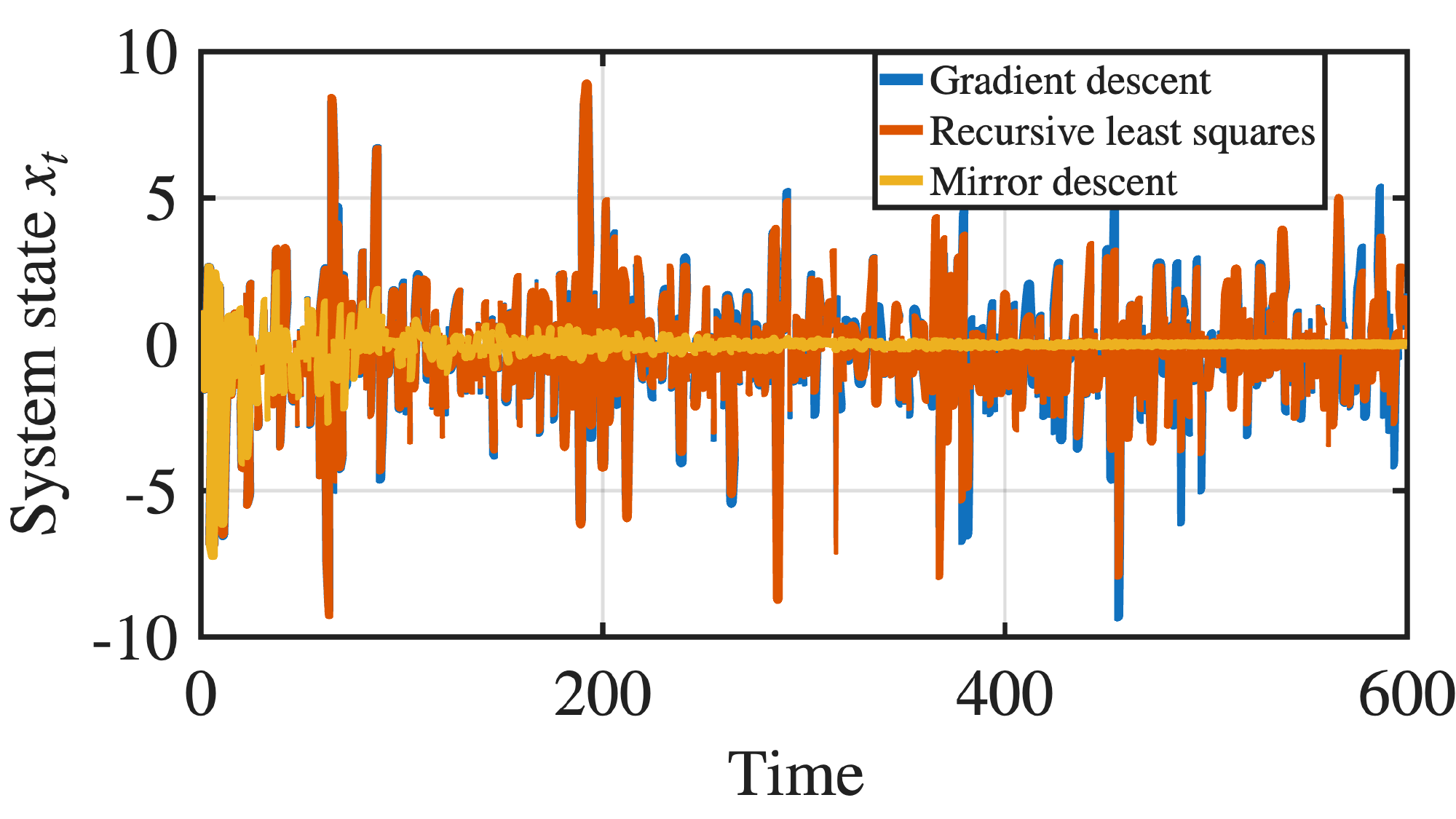}
        \label{fig:par_2000}
    }
    \caption{System state trajectories under control law \eqref{eq:motivating.controller} for varying dimensions $k$. Comparison between gradient descent \eqref{eq:motivating.GD}, recursive least squares \eqref{eq:motivating.RLS}, and sparse mirror descent \eqref{eq:sparse.update}.}
    \label{fig:par_all}
\end{figure*}
In this section, we present a motivating example to illustrate the behavior of commonly used Euclidean-based adaptive update laws, namely, gradient-descent-type and recursive least squares methods, in high-dimensional settings, and to highlight the need for non-Euclidean adaptation. To this end, consider the adaptive stabilization problem for the scalar system
\begin{equation}\label{eq:motivating.system}
    x_{t+1} = \theta^\top \psi(x_t, t) + u_t,
\end{equation}
where $\theta \in \mathbb{R}^k$ is an unknown parameter vector. We assume $\psi(x_t, t)$ is a stochastic regressor vector whose components $\psi_i(x_t, t)$ are independent and identically distributed (i.i.d.) processes defined by:
\begin{equation*}
    \psi_i(x_t, t) = 
    \begin{cases} 
        0.5 x_t + 1, & \text{with probability } 0.5, \\
        -0.5 x_t - 1, & \text{with probability } 0.5.
    \end{cases}
\end{equation*}
We further assume that $\theta$ is sparse, with only the first three components being non-zero: $\theta_1 = \theta_2 = \theta_3 = 1$, and $\theta_i = 0$ for all $i \geq 4$. Consequently, increasing the dimension does not alter the underlying system dynamics, as it merely appends zero-valued parameters. Following the control structure in \eqref{eq:control.law}, we employ the certainty-equivalence control law:
\begin{equation}\label{eq:motivating.controller}
    u(x_t, \widehat{\theta}_t) = -\widehat{\theta}_t^\top \psi(x_t, t).
\end{equation}
The instantaneous cost $J_t(\widehat{\theta}_t)$ defined in \eqref{eq:cost} satisfies
\begin{equation}
    J_t(\widehat{\theta}_t) = \frac{1}{2} \left( (\theta - \widehat{\theta}_t)^\top \psi(x_t, t) \right)^2 = \frac{1}{2} x_{t+1}^2,
\end{equation}
with the gradient given by $\nabla J_t(\widehat{\theta}_t) = -x_{t+1} \psi(x_t, t)$. 

We consider two standard Euclidean update laws for the parameter
estimate $\widehat{\theta}_t$: (normalized) gradient descent and recursive
least squares. Under the Assumption~\ref{ass:linear-growth-regressor}, both update laws guarantee boundedness of the closed-loop signals and asymptotic convergence of $x_t$ to the origin; see, e.g., \cite{goodwin2009adaptive,astrom2008adaptive,landau2011adaptive}.

\textbf{(i) (Normalized) gradient descent.} At time step $t+1$, the estimate $\widehat{\theta}_t$ is updated according to
    \begin{equation}\label{eq:motivating.GD}
        \widehat{\theta}_{t+1} = \widehat{\theta}_t + \frac{\psi(x_t, t) x_{t+1}}{1 + \|\psi(x_t, t)\|_2^2}.
    \end{equation}

 \textbf{(ii) Recursive least squares.}  At time step $t+1$, the parameter estimate $\widehat{\theta}_t$ and the covariance matrix $P_t \in \mathbb{S}^{k }$ are updated according to
 \begin{equation}
    \begin{split}\label{eq:motivating.RLS}
        \widehat{\theta}_{t+1} &= \widehat{\theta}_t + \frac{P_t \psi(x_t, t) x_{t+1}}{1 + \psi(x_t, t)^\top P_t \psi(x_t, t)}, \\
        P_{t+1} &= P_t - \frac{P_t \psi(x_t, t) \psi(x_t, t)^\top P_t}{1 + \psi(x_t, t)^\top P_t \psi(x_t, t)}.
    \end{split}
\end{equation}
   It is worth noting that a key drawback of recursive least squares, relative to normalized gradient descent, is that it requires maintaining and updating a $k\times k$ matrix at each iteration, which may be impractical in high-dimensional problems.

Figure \ref{fig:par_all} illustrates the state trajectories for the closed-loop system using the aforementioned update laws alongside the proposed mirror descent algorithm \eqref{eq:update.law}. The mirror descent scheme utilizes the mirror map  \eqref{sparse.potential}. All algorithms are initialized with $\widehat{\theta}_0 = \mathbf{0}$ and $x_0 = 1$, and simulations are conducted for $k \in \{10, 50, 500, 2000\}$. As observed, although the underlying system dynamics remain unchanged with increasing dimension, the performance of normalized gradient descent and recursive least squares degrades significantly. This behavior is consistent with the regret bound in \eqref{eq:intro.regret.bound}. Intuitively, by propagating estimation errors across all dimensions, Euclidean-based methods fail to isolate the low-dimensional signal. In contrast, mirror descent maintains robust, nearly dimension-invariant performance by exploiting the underlying sparse geometry of the parameter space.

\section{Mirror updates for discrete-time adaptive control}\label{sec:main.result}
This section develops a mirror descent framework for parameter adaptation in adaptive control of systems of the form \eqref{eq.dynamics}. We establish the closed-loop stability and asymptotic convergence properties of the resulting adaptive system. Furthermore, we derive a regret bound to characterize the transient performance of the proposed  scheme.

\subsection{Preliminaries}\label{sec:main.prel}
This subsection reviews preliminary definitions and results from convex analysis that will be used throughout the subsequent development. For a comprehensive treatment, we refer the reader to \cite{rockafellar1970convex, borwein2006convex,bertsekas2003convex, boyd2004convex}.

We consider the vector space of matrices \(\mathbb{R}^{m \times k}\) equipped with the Frobenius inner product
\begin{equation*} \langle X,Y\rangle := \operatorname{Tr}(X^\top Y), \qquad X,Y\in\mathbb{R}^{m\times k}. 
\end{equation*}
Let \(\|\cdot\|\) be a norm on \(\mathbb{R}^{m \times k}\). Its dual norm, denoted by \(\|\cdot\|_*\), is defined for any \(Z \in \mathbb{R}^{m \times k}\) by
\begin{equation*}
    \|Z\|_* := \sup_{X \in \mathbb{R}^{m\times k}}\{ \langle Z, X \rangle : \|X\| \le 1 \}.
\end{equation*}
Let \(f : \mathbb{R}^{m \times k} \to \mathbb{R}\cup\{+\infty\}\) be a proper, lower semicontinuous, convex function, with effective domain
\begin{equation*}
    \operatorname{dom} f := \{ X \in \mathbb{R}^{m \times k} : f(X) < +\infty \}.
\end{equation*}
The \emph{Fenchel conjugate} of \(f\) is defined by
\begin{equation}
    f^*(Z) := \sup_{X \in \operatorname{dom} f}
    \bigl\{ \langle Z, X \rangle - f(X) \bigr\}.
\end{equation}
For \(Q \in \operatorname{dom} f\) such that \(\partial f(Q)\neq\varnothing\), and for a fixed subgradient \(Z \in \partial f(Q)\), we define the generalized \emph{Bregman divergence} associated with \(f\) by
\begin{equation}
    D_f(P,Q;Z)
    :=
    f(P) + f^*(Z) - \langle P, Z \rangle,
    \quad P \in \operatorname{dom} f.
\end{equation}
Since \(Z \in \partial f(Q)\), the Fenchel--Young equality implies that
\begin{equation*}
    f(Q) + f^*(Z) = \langle Q, Z \rangle.
\end{equation*}
Therefore,
$
    D_f(P,Q;Z)
    =
    f(P) - f(Q) - \langle P-Q, Z \rangle.
$
By the Fenchel--Young inequality, \(D_f(P,Q;Z)\ge 0\). Geometrically, \(D_f(P,Q;Z)\) measures the gap between \(f(P)\) and the supporting affine function of \(f\) at \(Q\) induced by subgradient~\(Z\).
A convex function \(f\) is \(\mathcal{L}\)-smooth with respect to \(\|\cdot\|\) if it is differentiable and \begin{equation*} f(P) \leq f(Q) + \langle \nabla f(Q), P-Q \rangle + \frac{\mathcal{L}}{2}\|P-Q\|^2, \qquad \forall P,Q \in \operatorname{dom} f. \end{equation*} A proper convex function \(f\) is \(\mu\)-strongly convex with respect to \(\|\cdot\|\) if \begin{equation*} f(P) \geq f(Q) + \langle Z, P-Q \rangle + \frac{\mu}{2}\|P-Q\|^2, \qquad \forall P,Q \in \operatorname{dom} f,\quad \forall Z \in \partial f(Q). \end{equation*} 
Throughout the paper, we assume that the mirror map \(f\) in the update law \eqref{eq:update.law} is \(\mu\)-strongly convex with respect to \(\|\cdot\|\), but not necessarily smooth. Under this assumption, its Fenchel conjugate \(f^*\) is differentiable and \((1/\mu)\)-smooth with respect to the dual norm \(\|\cdot\|_*\). Moreover, the inverse relation
\begin{equation*}
    \nabla f^* = (\partial f)^{-1}
\end{equation*}
holds, where \(\partial f\) denotes the subdifferential of \(f\). Lastly, for the linear operator \(\mathcal{B}:\mathbb{R}^{m\times n_2}\to \mathbb{R}^{n_1\times n_2}\), its adjoint \(\mathcal{B}^{\top}:\mathbb{R}^{n_1\times n_2}\to \mathbb{R}^{m\times n_2}\) is defined through the identity \begin{equation}\label{eq:adj.ident} \langle \mathcal{B}(U),V\rangle = \langle U,\mathcal{B}^{\top}(V)\rangle, \end{equation} for all \(U\in\mathbb{R}^{m\times n_2}\) and \(V\in\mathbb{R}^{n_1\times n_2}\).

\subsection{Closed-loop stability and finite-time performance}
Before presenting the main stability results, we detail the algorithmic steps of the parameter update law \eqref{eq:update.law.main}. Given an initial estimate $\widehat{\Theta}_0$, we select a compatible dual initialization $Z_0 \in \partial f(\widehat{\Theta}_0)$. The adaptation then evolves in the dual space according to the gradient-descent recursion
$
    Z_{t+1} = Z_t - \eta_{t+1}\nabla J_t(\widehat{\Theta}_t).
$
At each time step, the primal parameter estimate is recovered via the mapping
$
    \widehat{\Theta}_{t+1} = \nabla f^*(Z_{t+1}).
$
As discussed in the previous subsection, since $f$ is $\mu$-strongly convex, its
Fenchel conjugate $f^*$ is differentiable and satisfies
\(
    \nabla f^* = (\partial f)^{-1}.
\)
Consequently, the primal--dual consistency relation
\(
    Z_{t+1}\in \partial f(\widehat{\Theta}_{t+1})
\)
is preserved by construction at each time step.
We emphasize that the use of nondifferentiable mirror maps provides a convenient
mechanism for incorporating constraints on $\Theta$. Specifically, the indicator
function of the feasible set can be included directly in the definition of $f$.
The resulting extended-valued mirror map enforces feasibility through its effective domain while preserving the primal--dual update structure described above; see Section~\ref{sec:simplex}.

The following theorem establishes the stability of
the closed-loop system and the asymptotic convergence of the tracking error to zero.

\begin{theorem}[Stability and convergence]\label{Theorem.stability} Consider the system \eqref{eq.dynamics} under the control law \eqref{eq:control.law} and the parameter update law \eqref{eq:update.law}. Suppose that Assumption~\ref{ass:linear-growth-regressor} holds. Then, the sequences \(\{X_t\}_{t\geq 0}\) and \(\{\widehat{\Theta}_t\}_{t\geq 0}\) are uniformly bounded. Moreover, 
\begin{equation}
   \lim_{t\to\infty} \|X_t-X_t^{\mathrm d}\|_{\mathrm F}=0.   
\end{equation}
\end{theorem}

\begin{proof} 
Consider the nonnegative  function
$
    V_t := D_f (\Theta, \widehat{\Theta}_t; Z_t).
$
The increment of $V_t$ is given by
\begin{align*}
\Delta V_{t+1} := V_{t+1} - V_t = f^*(Z_{t+1}) - f^*(Z_t) - \langle \Theta, Z_{t+1} - Z_t \rangle.
\end{align*}
Since  $f^*$ is $(1/\mu)$-smooth with respect to $\|\cdot\|_*$, we have
\begin{equation*}
f^*(Z_{t+1}) \le f^*(Z_t) + \langle \nabla f^*(Z_t), Z_{t+1} - Z_t \rangle + \frac{1}{2\mu}\|Z_{t+1} - Z_t\|_*^2.
\end{equation*}
Substituting this bound into the expression for $\Delta V_{t+1}$ yields
\begin{align*}
\Delta V_{t+1} \le \langle \nabla f^*(Z_t) - \Theta, \, Z_{t+1} - Z_t \rangle + \frac{1}{2\mu}\|Z_{t+1} - Z_t\|_*^2.
\end{align*}
Substituting the update law \eqref{eq:update.law.main} into the inequality above gives
\begin{align} \label{eq:proof.ineq.En}
     \Delta V_{t+1} \leq\eta_{t+1} \langle  \nabla J_t(\widehat{\Theta}_t), \Theta-\widehat{\Theta}_t  \rangle +  \frac{\eta_{t+1}^2}{2\mu} \| \nabla J_t(\widehat{\Theta}_t)\|^2_*
\end{align}
By exploiting the adjoint identity \eqref{eq:adj.ident}, the inner product term can be evaluated as follows:
\begin{align}\label{eq:key.inner.product}
\langle \nabla J_t(\widehat{\Theta}_t), \, \Theta - \widehat{\Theta}_t \rangle 
&= - \langle \mathcal{B}^\top(\widetilde{X}_{t+1}) \Psi(X_t,t)^\top, \, \Theta - \widehat{\Theta}_t \rangle = - \langle \mathcal{B}^\top(\widetilde{X}_{t+1}), \, (\Theta - \widehat{\Theta}_t)\Psi(X_t,t) \rangle \nonumber\\
&= - \langle \widetilde{X}_{t+1}, \, \mathcal{B} \big( (\Theta - \widehat{\Theta}_t) \Psi(X_t,t) \big) \rangle = - \langle \widetilde{X}_{t+1}, \, \widetilde{X}_{t+1} \rangle 
= - \|\widetilde{X}_{t+1}\|_\mathrm{F}^2 .
\end{align}
Substituting this result back into \eqref{eq:proof.ineq.En} yields
\begin{align} \label{eq:proof.ineq.tildex}
\Delta V_{t+1} \le -\eta_{t+1} \|\widetilde{X}_{t+1}\|_\mathrm{F}^2 + \frac{\eta_{t+1}^2}{2\mu} \| \nabla J_t(\widehat{\Theta}_t)\|_*^2.
\end{align}
Substituting the step size \eqref{eq:step-size} into the inequality above yields
\begin{align} \label{eq:proof.ineq.tildex.2}
\Delta V_{t+1} \le -\frac{\mu}{2} \frac{\|\widetilde{X}_{t+1}\|_\mathrm{F}^4}{\epsilon_{t+1} + \| \nabla J_t(\widehat{\Theta}_t)\|_*^2} \le 0.
\end{align}
Equation \eqref{eq:proof.ineq.tildex.2} implies that the nonnegative sequence $\{V_t\}_{t \ge 0}$ is monotonically nonincreasing and converges to a finite asymptotic value. Moreover, since \(f\) is \(\mu\)-strongly convex and \(Z_t\in\partial f(\widehat{\Theta}_t)\), we have
\begin{align*}
    \frac{\mu}{2}\|\Theta-\widehat{\Theta}_t\|^2 \leq V_t \leq V_0.
\end{align*}
Hence, the estimate $\widehat{\Theta}_t$ remains bounded. Since \(V_t\) converges, we also have $\lim_{t \to \infty} \Delta V_{t+1} = 0$. Thus, \eqref{eq:proof.ineq.tildex.2} implies
\begin{equation}\label{eq:limit_zero}
\lim_{t \to \infty} \frac{\|\widetilde{X}_{t+1}\|_\mathrm{F}^4}{\epsilon_{t+1} + \| \nabla J_t(\widehat{\Theta}_t)\|_*^2} = 0.
\end{equation}
The above limit does not, by itself, imply the boundedness or asymptotic convergence of $\widetilde{X}_t$, since the denominator may grow faster than the numerator. We next rule out this possibility under the linear-growth condition \eqref{eq:linear-growth-regressor}. To this end, we first establish an upper bound on the growth rate of the denominator in terms of the residual norm $\|\widetilde{X}_{t+1}\|_{\mathrm{F}}$.

\begin{lemma}\label{lemma.upper.bound} Suppose that $\Psi$ satisfies Assumption~\ref{ass:linear-growth-regressor}. Then there exist constants $c'_1,c'_2>0$ such that \begin{equation}\label{eq:denominator-upper-bound} \epsilon_{t+1} +\|\nabla J_t(\widehat{\Theta}_t)\|_*^2 \leq c'_1+c'_2\max_{0\leq i\leq t} \|\widetilde{X}_{i+1}\|_{\mathrm{F}}^4\qquad , \forall t \in \mathbb{Z}_{+}, . \end{equation} \end{lemma} The proof of Lemma~\ref{lemma.upper.bound} is provided in Appendix~\ref{sec:appendix}. Combining \eqref{eq:denominator-upper-bound} with \eqref{eq:limit_zero}, we obtain 
\begin{equation} 
0 \leq \frac{\|\widetilde{X}_{t+1}\|_{\mathrm{F}}^4} {c'_1+c'_2\displaystyle\max_{0\leq i\leq t} \|\widetilde{X}_{i+1}\|_{\mathrm{F}}^4} \nonumber\leq \frac{\|\widetilde{X}_{t+1}\|_{\mathrm{F}}^4} {\epsilon_{t+1} +\|\nabla J_t(\widehat{\Theta}_t)\|_*^2}. \label{eq:proof.lim} \end{equation} 
Hence, we have \begin{equation}\label{eq:normalized-error-limit} \lim_{t\to\infty} \frac{\|\widetilde{X}_{t+1}\|_{\mathrm{F}}^4} {c'_1+c'_2\displaystyle\max_{0\leq i\leq t} \|\widetilde{X}_{i+1}\|_{\mathrm{F}}^4} =0. \end{equation} Note that if $\{\widetilde{X}_{t+1}\}_{t \in \mathbb{Z}_{+}}$ is bounded, then \eqref{eq:normalized-error-limit} implies that $\|\widetilde{X}_{t+1}\|_{\mathrm{F}}\to 0$, since its denominator is uniformly bounded. We now establish the boundedness of $\{\widetilde{X}_{t+1}\}_{t \in \mathbb{Z}_{+}}$ using a contradiction argument analogous to the key technical lemma in discrete-time adaptive control \cite[Lemma 6.2.1]{goodwin2009adaptive}. Suppose, for contradiction, that the sequence $\{\|\widetilde{X}_{t+1}\|_{\mathrm{F}}\}_{t \in \mathbb{Z}_{+}}$ is unbounded. Then there exists a strictly increasing sequence of indices $\{t_n\}_{n\geq 1}$ such that \begin{equation*} \|\widetilde{X}_{t_n+1}\|_{\mathrm{F}} = \max_{0\leq i\leq t_n} \|\widetilde{X}_{i+1}\|_{\mathrm{F}} \quad\text{and}\quad \|\widetilde{X}_{t_n+1}\|_{\mathrm{F}}\to\infty \end{equation*} as $n\to\infty$. Therefore, along this subsequence

\begin{align*} \lim_{n\to\infty} \frac{\|\widetilde{X}_{t_n+1}\|_{\mathrm{F}}^4} {c'_1+c'_2\displaystyle\max_{0\leq i\leq t_n} \|\widetilde{X}_{i+1}\|_{\mathrm{F}}^4}  = \lim_{n\to\infty} \frac{\|\widetilde{X}_{t_n+1}\|_{\mathrm{F}}^4} {c'_1+c'_2\|\widetilde{X}_{t_n+1}\|_{\mathrm{F}}^4}  = \lim_{n\to\infty} \frac{1} {c'_1/\|\widetilde{X}_{t_n+1}\|_{\mathrm{F}}^4+c'_2} = \frac{1}{c'_2}>0. \end{align*} This contradicts \eqref{eq:normalized-error-limit}. Hence, $\{\widetilde{X}_{t+1}\}_{t \in \mathbb{Z}_{+}}$ is bounded. Consequently, the denominator in \eqref{eq:normalized-error-limit} is uniformly bounded, and therefore  $\widetilde{X}_{t+1} \to 0$ as $t \to \infty$.

Finally, the tracking error dynamics can be expressed as
\begin{equation*}
X_{t+1} - X^{\mathrm{d}}_{t+1} = (\mathcal{A} - \mathcal{B}\mathcal{K})(X_t - X^{\mathrm{d}}_t) + \widetilde{X}_{t+1}.
\end{equation*}
Since $\widetilde{X}_{t+1} \to 0$ as $t \to \infty$ and the closed-loop operator $(\mathcal{A} - \mathcal{B}\mathcal{K})$ is Schur stable, it follows from standard stability theory for linear systems under vanishing inputs (see e.g. \cite{hespanha2018linear}) that $X_t \to X^{\mathrm{d}}_t$ as $t \to \infty$.
\end{proof}

\begin{remark}[Non-Euclidean Polyak step size]\label{remark.stepsize}\normalfont
Since $J_t(\Theta)=0$, the step size in \eqref{eq:step-size} can be written as
\begin{equation*}
    \eta_{t+1}
    =
    \frac{2\mu\,\big(J_t(\widehat{\Theta}_t)-J_t(\Theta)\big)}
    {\epsilon_{t+1}+\|\nabla J_t(\widehat{\Theta}_t)\|_*^2}.
\end{equation*}
In the special case where the regressor is constant, i.e.,
$\Psi=C\in\mathbb{R}^{k\times n_2}$, the loss $J_t$ is time invariant.
Moreover, if the Euclidean geometry is used and $\epsilon_{t+1}=0$, then the
step size reduces to the classical Polyak step size
\cite{polyak1987introduction,hazan2019revisiting}. Therefore,
\eqref{eq:step-size} can be viewed as a non-Euclidean, online generalization of
the Polyak step size. The regularization term $\epsilon_{t+1}$ provides an additional degree of freedom to avoid potential numerical issues, such as division by zero in the presence of noise. A typical choice is $\epsilon_{t+1}=\alpha(t+1)^{-\beta}$ with $\alpha,\beta \in \mathbb{R}_{+}$.
\end{remark}

In the following theorem, we establish a regret bound for the proposed framework to characterize its finite-time (non-asymptotic) performance.

\begin{theorem}[Regret bound]
\label{Theorem.regret}
Consider the system \eqref{eq.dynamics} under the control law
\eqref{eq:control.law} with the parameter update law
\eqref{eq:update.law}. Then, for any horizon $T\geq 1$, the regret satisfies
\small
\begin{equation} \label{eq:regret.bound.theorem3}
       \operatorname{Reg}(T)
     \leq
     \sqrt{
     \frac{D_f(\Theta,\widehat{\Theta}_0;Z_0)}{2\mu}
    \sum_{t=0}^{T-1} \left(\epsilon_{t+1}+
        \| \nabla J_t(\widehat{\Theta}_t)\|_{*}^{2}
     \right)
     },
\end{equation}
\end{theorem}
\normalsize
\begin{proof}
Following the proof of Theorem~\ref{Theorem.stability}, define $ V_t := D_f(\Theta,\widehat{\Theta}_t;Z_t). $ Summing both sides of \eqref{eq:proof.ineq.tildex.2} over \(t=0,\ldots,T-1\) yields
\small
\begin{align}
    \sum_{t=0}^{T-1} \frac{J^2_t(\widehat{\Theta}_t)}{\epsilon_{t+1} + \| \nabla J_t(\widehat{\Theta}_t)\|_{*}^2} = \frac{1}{4}\sum_{t=0}^{T-1}\frac{\|\widetilde{X}_{t+1}\|_\mathrm{F}^4}{\epsilon_{t+1} + \| \nabla J_t(\widehat{\Theta}_t)\|_{*}^2}\leq \frac{1}{2\mu} \sum_{t=0}^{T-1} (V_t - V_{t+1}) = \frac{1}{2\mu} (V_0-V_T)\leq \frac{V_0}{2\mu}. \label{eq:Theorem3.proof.sum}
\end{align}
\normalsize
By applying the Cauchy-Schwarz inequality, we obtain

\begin{align*}
   \operatorname{Reg}(T) = \sum_{t=0}^{T-1} J_t(\widehat{\Theta}_t) \nonumber =& \sum_{t=0}^{T-1} \frac{J_t(\widehat{\Theta}_t)}{\sqrt{\epsilon_{t+1} + \| \nabla J_t(\widehat{\Theta}_t)\|^2_{*}}}\, \sqrt{\epsilon_{t+1} + \| \nabla J_t(\widehat{\Theta}_t)\|^2_{*}} \nonumber \\
    &\leq \sqrt{ \sum_{t=0}^{T-1} \frac{J^2_t(\widehat{\Theta}_t)}{\epsilon_{t+1} + \| \nabla J_t(\widehat{\Theta}_t)\|^2_{*}}}\,\sqrt{ \sum_{t=0}^{T-1} (\epsilon_{t+1} + \| \nabla J_t(\widehat{\Theta}_t)\|^2_{*}) }.
\end{align*}
Substituting the upper bound from \eqref{eq:Theorem3.proof.sum} into the  above inequality results in
\begin{align*}
    \operatorname{Reg}(T) \leq \sqrt{\frac{V_0}{2\mu} \left( \sum_{t=0}^{T-1} (\epsilon_{t+1} + \| \nabla J_t(\widehat{\Theta}_t)\|^2_{*}) \right)},
\end{align*}
which completes the proof.
\end{proof}

\begin{remark}[Interpretation of the regret bound]\normalfont As discussed in the introduction, the choice of the mirror map \(f\) in the update law \eqref{eq:update.law.main} affects the regret bound \eqref{eq:regret.bound.theorem3} through the Bregman divergence \(D_f\), the dual norm \(\|\cdot\|_{*}\), and the strong convexity parameter \(\mu\). Since \(f\) may be normalized, without loss of generality, to be \(1\)-strongly convex, the quantities of primary interest are the Bregman divergence \(D_f\) and the dual norm \(\|\cdot\|_{*}\) associated with the norm relative to which \(f\) is strongly convex. The main idea underlying the dimension-independent regret bounds developed in the next two sections is to choose \(f\) and the underlying norm so that the corresponding dual norm does not introduce explicit dimension-dependent factors. Examples include the \(\ell_\infty\) norm, the spectral norm, and their appropriate analogues. The dimension dependence of the Bregman term \(D_f(\Theta,\widehat{\Theta}_0;Z_0)\) can then be controlled through a suitable initialization, such as zero initialization when exploiting sparse structure.
\end{remark}

\section{Sparse Structure on $\Theta$}\label{sec:sparse}
In this section, we assume a given sparsity structure on the unknown matrix $\Theta \in \Omega= \mathbb{R}^{m \times k}$. We consider three settings: entrywise sparsity, group sparsity, and low-rank structure. In each case, we first specify an appropriate mirror map function and derive the explicit form of the update law \eqref{eq:update.law.main}. Moreover, using the general regret bound established in Theorem~\ref{Theorem.regret}, we derive corresponding bounds for each sparsity structure.
\subsection{Entrywise sparsity.}

We consider the case where the unknown matrix $\Theta \in  \mathbb{R}^{m\times k}$ is entrywise $s$-sparse, i.e., it has at most $s$ nonzero entries. In this setting, we equip $\mathbb{R}^{m\times k}$ with the entrywise $\ell_1$-norm,
$
\|\cdot\| = \|\cdot\|_1,
$
whose dual norm is the entrywise $\ell_\infty$-norm,
$
\|\cdot\|_* = \|\cdot\|_\infty. 
$ Motivated by proximal setups for $\ell_1$-based sparse recovery \cite{nesterov2013first,juditsky2023sparse}, we use the
update law \eqref{eq:update.law} with the mirror map
\begin{equation}\label{sparse.potential}
 f_{\ell_p}(\widehat{\Theta}) = \frac{d^{2-2/p}}{2(p-1)}  \|\widehat{\Theta}\|_p^2,
\qquad
p = 1 + \frac{1}{\ln(d)}.   
\end{equation}
Note that, for $1<p\le 2$, the function $\widehat{\Theta}\mapsto \frac12\|\widehat{\Theta}\|_p^2$ is $(p-1)$-strongly convex with respect to the norm $\|\cdot\|_p$. Moreover, since
$
\|\widehat{\Theta}\|_p\ge d^{\frac1p-1}\|\widehat{\Theta}\|_1,
$ it follows that $f_{\ell_p}$ is $1$-strongly convex with respect to the
entrywise $\ell_1$-norm. The gradient of $f_{\ell_p}$ is given by
\begin{align*}
\nabla f_{\ell_p}(\widehat{\Theta})&=  \frac{d^{2-2/p}}{(p-1)} 
    \|\widehat{\Theta}\|_p^{\,2-p}\,
\operatorname{sign}(\widehat{\Theta})\odot |\widehat{\Theta}|^{\,p-1}.
\end{align*}
Here, $\odot$ denotes the Hadamard product. Let $q$ be the dual exponent of $p$, defined by $ q := p/(p-1)=1+\ln(d)$.  The inverse gradient map is
\begin{align*}
\nabla f^*(Z)= (\nabla f_{\ell_p})^{-1}(Z)&= \frac{p-1}{d^{2-2/p}}
 \|Z\|_{q}^{\,2-q}
\operatorname{sign}(Z)\odot |Z|^{\,q-1}
\end{align*}
As a result, the update law \eqref{eq:update.law} can be written explicitly in the primal space as
\begin{tcolorbox}[ colback=black!2, colframe=black, boxrule=0.6pt, arc=0pt, left=6pt, right=6pt, top=4pt, bottom=4pt ]
\begin{equation} \label{eq:sparse.update} \begin{split} \widehat{\Theta}_{t+1} &= \nabla f^*_{\ell_p} \!\left( \nabla f_{\ell_p}(\widehat{\Theta}_t) -\eta_{t+1}\nabla J_t(\widehat{\Theta}_t) \right), \\ \eta_{t+1} &= \frac{2\,J_t(\widehat{\Theta}_t)} {\epsilon_{t+1} +\|\nabla J_t(\widehat{\Theta}_t)\|_\infty^2}. \end{split} \end{equation}
\end{tcolorbox}
\noindent Applying the general regret bound in
\eqref{eq:regret.bound.theorem3} to the mirror map
\eqref{sparse.potential} yields the following result.
\begin{corollary}[Regret bound under entrywise sparsity] Consider the system \eqref{eq.dynamics} under the control law \eqref{eq:control.law} and the update law \eqref{eq:sparse.update}. Suppose that the parameter estimate is initialized as \(\widehat{\Theta}_0=0\) and that the system parameter matrix \(\Theta\) has at most \(s\) nonzero entries. Then,
\small
\begin{align*}
 \operatorname{Reg}(T)
\le \frac{e}{2}
 s^{\frac{\ln(d)}{1+\ln(d)} } \|\Theta\|_\infty
\sqrt{\ln(d)
 \sum_{t=0}^{T-1} \left(\epsilon_{t+1} + \| \nabla J_t(\widehat{\Theta}_t)\|^2_{\infty}\right)
}.
\end{align*}
\normalsize

\end{corollary}

\begin{proof}
Since \(\widehat{\Theta}_0=0\), we have
\begin{align*}
D_f(\Theta,\widehat{\Theta}_0;Z_0)
=D_f(\Theta,0;0)
=f(\Theta) =\frac{d^{2-2/p}}{2(p-1)}\|\Theta\|_p^2,
\end{align*}
where we used \(f(0)=0\) and \(\nabla f(0)=0\).
Moreover, since \(\Theta\) has at most \(s\) nonzero entries,
\(\|\Theta\|_p\leq s^{1/p}\|\Theta\|_\infty\). Therefore,
\begin{align*}
D_f(\Theta,\widehat{\Theta}_0;Z_0)
\leq
\frac{d^{2-2/p}}{2(p-1)}
s^{2/p}\|\Theta\|_\infty^2 \leq
\frac{e^2\ln(d)}{2}
s^{\frac{2\ln(d)}{1+\ln(d)}}
\|\Theta\|_\infty^2,
\end{align*}
where the second inequality follows by setting
\(p=1+1/\ln(d)\) and noting that
\(d^{2-2/p}\leq e^2\).

Applying the general regret bound in
\eqref{eq:regret.bound.theorem3} with \(\mu=1\) and
\(\|\cdot\|_*=\|\cdot\|_\infty\) yields
\begin{align*}
\operatorname{Reg}(T)
&\leq
\sqrt{
\frac{D_f(\Theta,\widehat{\Theta}_0;Z_0)}{2\mu}
\sum_{t=0}^{T-1}
\left(
\epsilon_{t+1}
+
\|\nabla J_t(\widehat{\Theta}_t)\|_*^2
\right)
} \leq
\frac{e}{2}
s^{\frac{\ln(d)}{1+\ln(d)}}
\|\Theta\|_\infty
\sqrt{
\ln(d)
\sum_{t=0}^{T-1}
\left(
\epsilon_{t+1}
+
\|\nabla J_t(\widehat{\Theta}_t)\|_\infty^2
\right)
},
\end{align*}
which proves the claim.
\end{proof}

\allowdisplaybreaks
\subsection{Group sparsity}

We consider the case where the unknown matrix
\(
\Theta \in \mathbb{R}^{m\times k}
\)
has a \emph{group-sparse} structure.
Specifically, the entries of $\Theta$ are partitioned into $d_{\mathrm{grp}}$
 disjoint groups
\(
\Theta_1,\Theta_2,\dots,\Theta_{d_{\mathrm{grp}}},
\)
and $\Theta$ is said to be \emph{$s$-group sparse} if at most $s$ groups are nonzero, i.e.,
\(
\bigl|\{i : \Theta_i \neq 0\}\bigr| \le s.
\)
In this setting, we equip $\mathbb{R}^{m\times k}$ with the mixed
$\ell_{1}/\ell_{2}$ norm
\(
\|\Theta\|_{1,2}
:=
\sum_{i=1}^{d_{\mathrm{grp}}} \|\Theta_i\|_2,
\)
whose dual norm is
\(
\|Z\|_{\infty,2}
:=
\max_{1\le i\le d_{\mathrm{grp}}}\|Z_i\|_2.
\) We apply the update law \eqref{eq:update.law} with the block mirror map
\begin{equation}\label{eq:group.potential}
f_{p,2}(\widehat{\Theta})
:=
 \frac{d_{\rm grp}^{2-2/p}}{2(p-1)}\|\widehat{\Theta}\|_{p,2}^2,
\qquad
p = 1 + \frac{1}{\ln (d_{\mathrm{grp}}) },
\end{equation}
where
\(
\|\widehat{\Theta}\|_{p,2} := \left( \sum_{i=1}^{d_{\mathrm{grp}}} \|\widehat{\Theta}_{\,i}\|_2^p \right)^{1/p}.
\) Following the same line of reasoning as in the previous subsection, it follows that $f_{p,2}$ is $1$-strongly convex with respect to the norm $\|\cdot\|_{1,2}$. The gradient of $f_{p,2}$ is given blockwise by
\begin{equation*}
\bigl(\nabla f_{p,2}(\widehat{\Theta})\bigr)_i
= \frac{d_{\rm grp}^{2-2/p}}{(p-1)}
\|\widehat{\Theta}\|_{p,2}^{\,2-p}\,
\|\widehat{\Theta}_i\|_2^{\,p-2}\,
\widehat{\Theta}_i,
\;\, i=1,\dots,d_{\mathrm{grp}}.
\end{equation*}
Let $q$ be the dual exponent of $p$. The inverse gradient  map is
\begin{equation*}
\bigl(\nabla f_{p,2}^*(Z)\bigr)_i
= \frac{(p-1)}{d_{\rm grp}^{2-2/p}}
\|Z\|_{q,2}^{\,2-q}\,
\|Z_i\|_2^{\,q-2}\,
Z_i,
\;\;\, i=1,\dots,d_{\mathrm{grp}}.
\end{equation*}
Therefore, the update law \eqref{eq:update.law} can be written equivalently as
\begin{tcolorbox}[ colback=black!2, colframe=black, boxrule=0.6pt, arc=0pt, left=6pt, right=6pt, top=4pt, bottom=4pt ]
\begin{equation} \label{eq:update.group.sparse}
 \begin{split}
 \widehat{\Theta}_{t+1}
 &=
 \nabla f_{p,2}^*\!\left(
 \nabla f_{p,2}(\widehat{\Theta}_t)
 -
 \eta_{t+1} \nabla J_t(\widehat{\Theta}_t)
 \right),\\
 \eta_{t+1} &= \frac{2\,J_t(\widehat{\Theta}_t)} {\epsilon_{t+1} +\|\nabla J_t(\widehat{\Theta}_t)\|_{\infty,2}^2}.
  \end{split}
 \end{equation}
 \end{tcolorbox}
\noindent The following corollary provides a regret bound for the group sparsity structure that depends only logarithmically on $d_{\mathrm{grp}}$.
\begin{corollary}[Regret bound for group sparsity]
Consider the system \eqref{eq.dynamics} under the control law
\eqref{eq:control.law} and the update law \eqref{eq:update.group.sparse}.
Suppose that the update is initialized at $\widehat{\Theta}_0=0$ and the system parameter matrix $\Theta$ is $s$-group sparse.
Then the regret satisfies
\footnotesize
\begin{equation*}
\begin{split}
 \operatorname{Reg}(T)
&\le
\frac{e}{2} s^{\frac{\ln (d_{\mathrm{grp}})}{1+\ln (d_{\mathrm{grp}})}} \|\Theta\|_{\infty,2}\,
\sqrt{
\ln (d_{\mathrm{grp}})  \sum_{t=0}^{T-1} \left(\epsilon_{t+1} + \| \nabla J_t(\widehat{\Theta}_t)\|^2_{\infty,2}\right)
}.
\end{split}
\end{equation*}
\normalsize
\end{corollary}

\begin{proof} Since the update is initialized at $\widehat{\Theta}_0=0$, we have 
\begin{align*} D_f(\Theta,\widehat{\Theta}_0;Z_0) = f_{p,2}(\Theta) = \frac{d_{\rm grp}^{2-2/p}}{2(p-1)} \|\Theta\|_{p,2}^2 \leq \frac{d_{\rm grp}^{2-2/p}}{2(p-1)} s^{2/p}\|\Theta\|_{\infty,2}^2 \leq \frac{e^2}{2}\ln(d_{\rm grp}) s^{\frac{2\ln(d_{\rm grp})}{1+\ln(d_{\rm grp})}} \|\Theta\|_{\infty,2}^2, \end{align*}
where the first inequality follows from the $s$-group sparsity of $\Theta$. As a result, employing the general regret bound in \eqref{eq:regret.bound.theorem3} with \(\mu=1\) and
\(\|\cdot\|_*=\|\cdot\|_{\infty,2}\) yields
\small
\begin{align*}
 \operatorname{Reg}(T)
\le&
\sqrt{
\frac{D_f(\Theta,\widehat{\Theta}_0;Z_0)}{2 \mu}
\sum_{t=0}^{T-1} \left(\epsilon_{t+1} + \| \nabla J_t(\widehat{\Theta}_t)\|^2_{*}\right)
}
\le
\frac{e}{2} s^{\frac{\ln (d_{\mathrm{grp}})}{1+\ln (d_{\mathrm{grp}})}} \|\Theta\|_{\infty,2}\,
\sqrt{
\ln (d_{\mathrm{grp}})  \sum_{t=0}^{T-1} \left(\epsilon_{t+1} + \| \nabla J_t(\widehat{\Theta}_t)\|^2_{\infty,2}\right)
}.
\end{align*}
\normalsize
which completes the proof.
\end{proof}

\subsection{Low-rank matrix structure.}



In this subsection, we consider the case where the unknown matrix $\Theta\in\mathbb{R}^{m\times k}$ is low rank. Specifically, we assume that
\(
\operatorname{rank}(\Theta)\le s.
\)
Let
\(
    d_{\mathrm{sv}} := \min(m,k),
\)
which is the number of singular values of a matrix in $\mathbb{R}^{m\times k}$.
Our goal is to develop an update law whose regret
depends on the intrinsic rank $s$, rather than  on the
ambient singular-value dimension $d_{\mathrm{sv}}$. To this end, we equip $\mathbb{R}^{m\times k}$ with the nuclear norm
\(
    \|\cdot\| = \|\cdot\|_{\mathrm{S}_1},
\)
whose dual norm is the spectral norm
\(
    \|\cdot\|_* = \|\cdot\|_{\mathrm{S}_\infty}.
\)
Motivated by the sparse-vector construction in the previous subsection, we use
the update law \eqref{eq:update.law} with the Schatten-$p$ mirror map
\begin{equation}\label{lowrank.potential}
    f_{\mathrm{S}_p}(\widehat{\Theta})
    =
    \frac{d_{\mathrm{sv}}^{2-2/p}}{2(p-1)} \|\widehat{\Theta}\|_{\mathrm{S}_p}^2,
    \qquad
    p=1+\frac{1}{\ln(d_{\mathrm{sv}})} .
\end{equation}
The function $ f_{\mathrm{S}_p}$ is $1$-strongly convex with respect to the nuclear norm $\|\cdot\|_{\rm S_1}$.
Let
\(
\widehat{\Theta} = U\Sigma V^\top
\)
be a singular value decomposition (SVD) of $\widehat{\Theta}$. The gradient of
$f_{\mathrm{S}_p}$ is given by
\begin{align*}
   \nabla f_{\mathrm{S}_p}(\widehat{\Theta})
&=\frac{d_{\mathrm{sv}}^{2-2/p}}{(p-1)} \|\widehat{\Theta}\|_{\mathrm{S}_p}^{\,2-p}\; U\Sigma^{p-1}V^\top.
\end{align*}
Let $q$ be the dual exponent of $p$. For any $Z\in\mathbb{R}^{m\times k}$ with SVD
$Z=\bar U\bar\Sigma\bar V^\top$, the inverse gradient map is
\begin{align*}
    \nabla f_{\mathrm{S}_p}^*(Z)=(\nabla f_{\mathrm{S}_p})^{-1}(Z)
&=\frac{(p-1)}{d_{\mathrm{sv}}^{2-2/p}} \|Z\|_{\mathrm{S}_q}^{\,2-q}\;
\bar U \bar\Sigma^{\,q-1}\bar V^\top.
\end{align*}
Consequently, the parameter update \eqref{eq:update.law} can be written as
\begin{tcolorbox}[ colback=black!2, colframe=black, boxrule=0.6pt, arc=0pt, left=6pt, right=6pt, top=4pt, bottom=4pt ]
\begin{equation} \label{eq:update.lowrank.sparse}
 \begin{split}
 \widehat{\Theta}_{t+1}
 &=
 \nabla f_{\mathrm{S}_p}^*\!\left(
 \nabla f_{\mathrm{S}_p}(\widehat{\Theta}_t)
 -
 \eta_{t+1} \nabla J_t(\widehat{\Theta}_t)
 \right),\\
 \eta_{t+1} &= \frac{2\,J_t(\widehat{\Theta}_t)} {\epsilon_{t+1} +\|\nabla J_t(\widehat{\Theta}_t)\|_{\mathrm{S}_\infty}^2}.
  \end{split}
 \end{equation}
 \end{tcolorbox}
\noindent The following corollary provides a regret bound for the low-rank structure that depends only logarithmically on $d_{\mathrm{sv}}$.

\begin{corollary} [Regret bound for low-rank structure] Consider the system \eqref{eq.dynamics} under the control law
\eqref{eq:control.law} and the update law \eqref{eq:update.lowrank.sparse}.
Suppose that the update is initialized at $\widehat{\Theta}_0=0$ and the system parameter matrix $\Theta$ is of rank at most $s$. Then,
\small
\begin{align*}
 \operatorname{Reg}(T)\!
\le \! \frac{e}{2}
 s^{\frac{\ln(d_{\mathrm{sv}})}{1+\ln(d_{\mathrm{sv}})} } \|\Theta\|_{\mathrm{S}_\infty}
\!\sqrt{\ln(d_{\mathrm{sv}})\!
\sum_{t=0}^{T-1}\! \left(\epsilon_{t+1} + \| \nabla J_t(\widehat{\Theta}_t)\|^2_{\mathrm{S}_\infty}\right)
}.
\end{align*}
\normalsize
\end{corollary}

\begin{proof}
Since the update is initialized at $\widehat{\Theta}_0=0$, we obtain
\begin{align*}
D_f(\Theta,\widehat{\Theta}_0;Z_0)=f(\Theta)&=\frac{d_{\mathrm{sv}}^{2-2/p}}{2(p-1)}\|\Theta\|_{\mathrm{S}_p}^2\leq \frac{e^2}{2}\ln(d_{\mathrm{sv}}) \|\Theta\|_{\mathrm{S}_p}^2
\end{align*}
Moreover, since $\Theta$ is of rank at most $s$, we find that
\(
\|\Theta\|_{\mathrm{S}_p}\le s^{1/p}\|\Theta\|_{\mathrm{S}_\infty},
\)
hence
\begin{align*}
D_f(\Theta,\widehat{\Theta}_0;Z_0)
\le
\frac{e^2}{2}\ln(d_{\mathrm{sv}}) s^{2/p}\|\Theta\|_{\rm S_\infty}^2=\frac{e^2}{2}\ln(d_{\mathrm{sv}}) s^{\frac{2\ln(  d_{\mathrm{sv}})}{1+\ln(  d_{\mathrm{sv}})} }\|\Theta\|_{\rm S_\infty}^2.
\end{align*}
As a result, employing the general regret bound in \eqref{eq:regret.bound.theorem3} with \(\mu=1\) and
\(\|\cdot\|_*=\|\cdot\|_{\rm S_{\infty}}\) yields
\small
\begin{align*}
 \operatorname{Reg}&(T)
\le
\sqrt{
\frac{D_f(\Theta,\widehat{\Theta}_0;Z_0)}{2\mu}
\sum_{t=0}^{T-1} \left(\epsilon_{t+1} + \| \nabla J_t(\widehat{\Theta}_t)\|^2_{*}\right)
}
\le \frac{e}{2}
\|\Theta\|_{\mathrm{S}_\infty}
s^{\frac{\ln(  d_{\mathrm{sv}})}{1+\ln(  d_{\mathrm{sv}})} }
\sqrt{\ln(  d_{\mathrm{sv}})
\sum_{t=0}^{T-1} \left(\epsilon_{t+1} + \| \nabla J_t(\widehat{\Theta}_t)\|^2_{\mathrm{S}_\infty}\right)
},
\end{align*}
as claimed.
\normalsize
\end{proof}
\section{Simplex-Type Sets and Entropic Potentials}\label{sec:simplex}

In this section, we consider the case in which the unknown parameter set $\Omega \subset \mathbb{R}^{m \times k} $
has a simplex-type structure. Specifically, we study two representative examples: the standard probability simplex and the spectraplex, that is, the set of density matrices. For each setting, we identify a suitable mirror map and derive the explicit form of the update law \eqref{eq:update.law}. We then apply the general regret bound of Theorem~\ref{Theorem.regret} to obtain the corresponding regret guarantees and show that they are essentially dimension-free.

\subsection{Simplex and Shannon Entropy}

We consider the case where the unknown matrix $\Theta$ belongs to the
entrywise probability simplex
\begin{equation}\label{eq:EntrywiseSimplex}
\Omega_{\mathrm e}
:=
\left\{
\widehat{\Theta} \in \mathbb{R}_{+}^{m\times k}
\;:\;
\|\widehat{\Theta}\|_1 = 1
\right\}.
\end{equation}
In this setting, we equip $\mathbb{R}^{m\times k}$ with the entrywise
$\ell_1$-norm, $\|\cdot\|=\|\cdot\|_1$, whose dual norm is the entrywise
$\ell_\infty$-norm, $\|\cdot\|_*=\|\cdot\|_\infty$. For the update law \eqref{eq:update.law}, we employ the Shannon entropy plus the indicator function of $\Omega_{\mathrm e}$ as the mirror map
\begin{equation}\label{eq:shannon.potential}
f_{\mathrm{Sh}}(\widehat{\Theta})
=
\sum_{i,j}
\bigl(\widehat{\Theta}_{ij}\ln (\widehat{\Theta}_{ij})-\widehat{\Theta}_{ij}\bigr)
+
\iota_{\Omega_{\mathrm e}}(\widehat{\Theta}),\nonumber
\end{equation}
where $\iota_{\Omega_{\mathrm e}}$ denotes the indicator function of
$\Omega_{\mathrm e}$, i.e.,
\begin{equation*}
\iota_{\Omega_{\mathrm e}}(\widehat{\Theta})
=
\begin{cases}
0, & \widehat{\Theta}\in\Omega_{\mathrm e},\\
+\infty, & \text{otherwise}.
\end{cases}
\end{equation*}
By Pinsker's inequality, $f_{\mathrm{Sh}}$ is $1$-strongly convex on
$\Omega_{\mathrm e}$ with respect to the $\ell_1$-norm
\cite{nemirovski2009robust,Nemirovski2004MirrorProx}. Moreover, for any
$\widehat{\Theta}\in
 \operatorname{ri}(\Omega_{\mathrm e})
 =
 \left\{
 \widehat{\Theta}_{ij}>0
 \;:\;
 \|\widehat{\Theta}\|_1 = 1
 \right\},$ its subdifferential
admits the characterization
\begin{equation}\label{eq:entropy.subdiff}
\partial f_{\mathrm{Sh}}(\widehat{\Theta})
=
\left\{
\ln_{\mathrm{el}}(\widehat{\Theta}) + \lambda \mathbf{1}
:\;
\lambda \in \mathbb{R}
\right\},
\end{equation}
where $\mathbf{1}$ denotes the all-ones matrix and $\ln_{\mathrm{el}}(\cdot)$ denotes the entrywise logarithm. The conjugate of $f_{\mathrm{Sh}}$ is given by
\begin{equation}\label{eq:entropy.conjugate}
f_{\mathrm{Sh}}^*(Z)
=
\ln\!\left(\sum_{i,j}\exp(Z_{ij})\right)+1,\nonumber
\end{equation}
which is everywhere differentiable, with gradient
\begin{equation}\label{eq:softmax.gradient}
\nabla f_{\mathrm{Sh}}^*(Z)
=
\frac{\exp_{\mathrm{el}}(Z)}{\sum_{i,j}\exp(Z_{ij})},\nonumber
\end{equation}
where $\exp_{\mathrm{el}}(\cdot)$ denotes the entrywise exponential. 

Applying the update law \eqref{eq:update.law}, any choice of
$Z_t\in\partial f_{\mathrm{Sh}}(\widehat{\Theta}_t)$ can be written as
\(
Z_t
=
\ln_{\mathrm{el}}(\widehat{\Theta}_t) + \lambda_t \mathbf{1}
\)
for some $\lambda_t\in\mathbb{R}$.
The dual update \eqref{eq:update.law.main} thus yields
\begin{equation*}
Z_{t+1}
=
\ln_{\mathrm{el}}(\widehat{\Theta}_t)
-
\eta_{t+1}  \nabla J_t(\widehat{\Theta}_t)
+
\lambda_t \mathbf{1}.
\end{equation*}
Applying the primal recovery $\widehat{\Theta}_{t+1}=\nabla f_{\mathrm{Sh}}^*(Z_{t+1})$ yields the following parameter update law
\begin{tcolorbox}[ colback=black!2, colframe=black, boxrule=0.6pt, arc=0pt, left=6pt, right=6pt, top=4pt, bottom=4pt ]
\begin{equation} \label{eq:shanon.update}
\begin{split}
\widehat{\Theta}_{t+1}
&=
\frac{
\widehat{\Theta}_t \odot \exp_{\mathrm{el}}(-\eta_{t+1} \nabla J_t(\widehat{\Theta}_t))
}{
\bigl\|\widehat{\Theta}_t \odot \exp_{\mathrm{el}}(-\eta_{t+1} \nabla J_t(\widehat{\Theta}_t))\bigr\|_1
},\\
\eta_{t+1} &= \frac{2\,J_t(\widehat{\Theta}_t)} {\epsilon_{t+1} +\|\nabla J_t(\widehat{\Theta}_t)\|_\infty^2}.
\end{split}
\end{equation}
\end{tcolorbox}
\noindent where $\odot$ denotes the Hadamard product.  Note that the above update law is independent of the scalar $\lambda_t$. The following corollary provides the resulting regret bound.
\begin{corollary} [Regret bound over the simplex]
Consider the system \eqref{eq.dynamics} under the control law
\eqref{eq:control.law} and the parameter update law
\eqref{eq:shanon.update}. Suppose that the update is initialized at
$\widehat{\Theta}_0 = \mathbf{1}/d$ and that the system parameter matrix belongs
to the simplex $\Omega_{\mathrm{e}}$ defined in \eqref{eq:EntrywiseSimplex}. Then, the regret satisfies
\begin{align}\label{eq:corollary.simplex}
 \operatorname{Reg}(T)&
\le 
\sqrt{\frac{\ln(d)}{2}
 \sum_{t=0}^{T-1} \left(\epsilon_{t+1} + \| \nabla J_t(\widehat{\Theta}_t)\|^2_{\infty}\right)
}.
\end{align}
\end{corollary}
\begin{proof}
Let \(Z_0\in\partial f_{\mathrm{Sh}}(\widehat\Theta_0)\). Since
\(\widehat\Theta_0=\mathbf 1/d\in\operatorname{ri}(\Omega_{\mathrm{e}})\), by
\eqref{eq:entropy.subdiff} we may write
\(Z_0=\ln_{\mathrm{el}}(\widehat\Theta_0)+\lambda_0\mathbf 1
= -\ln(d)\mathbf 1+\lambda_0\mathbf 1\) for some
\(\lambda_0\in\mathbb R\). Hence \(Z_0\) is a constant multiple of
\(\mathbf 1\). Since \(\Theta,\widehat\Theta_0\in\Omega_{\mathrm{e}}\), we have
\(\langle \Theta-\widehat\Theta_0,\mathbf 1\rangle=0\), and therefore
\(\langle \Theta-\widehat\Theta_0,Z_0\rangle=0\). It follows from the
definition of the generalized Bregman divergence that
\begin{align}
D_{f_{\mathrm{Sh}}}(\Theta,\widehat\Theta_0;Z_0)
=
f_{\mathrm{Sh}}(\Theta)-f_{\mathrm{Sh}}(\widehat\Theta_0)
-\langle \Theta-\widehat\Theta_0,Z_0\rangle &=
f_{\mathrm{Sh}}(\Theta)-f_{\mathrm{Sh}}(\widehat\Theta_0) \nonumber\\
&=
\sum_{i,j}\Theta_{ij}\ln(\Theta_{ij})+\ln(d) \le \ln(d),\label{eq:shanoon.proof.ineq}
\end{align}
where the last inequality follows from the convention \(0\ln 0=0\) and
the fact that \(x\ln x\le 0,\, \forall x\in[0,1]\).
Consequently, applying the general regret bound in
\eqref{eq:regret.bound.theorem3} with \(\mu=1\) and
\(\|\cdot\|_*=\|\cdot\|_\infty\), and using
\eqref{eq:shanoon.proof.ineq}, yields 
\begin{align*}
 \operatorname{Reg}(T)
&\le
\sqrt{
\frac{D_f(\Theta,\widehat{\Theta}_0;Z_0)}{2\mu}
\sum_{t=0}^{T-1} \left(\epsilon_{t+1} + \| \nabla J_t(\widehat{\Theta}_t)\|^2_{*}\right)
}\le \sqrt{\frac{\ln(d)}{2}
 \sum_{t=0}^{T-1} \left(\epsilon_{t+1} + \| \nabla J_t(\widehat{\Theta}_t)\|^2_{\infty}\right)}.
\end{align*}
which proves the claim.
\end{proof}
\begin{remark}[Row-stochastic simplex]\normalfont
The same construction applies when \(\Theta\) is row-stochastic, i.e.,
\begin{equation}\label{eq:RowStochasticSimplex}
\Omega_{\mathrm r}
:=
\left\{
\widehat{\Theta}\in\mathbb{R}_{+}^{m\times k}
:\;
\sum_{j=1}^{k}\widehat{\Theta}_{ij}=1,\ i=1,\ldots,m
\right\}.
\end{equation}
In this case, we use
\begin{equation}\label{eq:row.shannon.potential}
f_{\mathrm{Sh},r}(\widehat{\Theta})
=
\sum_{i=1}^{m}\sum_{j=1}^{k}
\left(\widehat{\Theta}_{ij}\ln(\widehat{\Theta}_{ij})-\widehat{\Theta}_{ij}\right)
+
\iota_{\Omega_{\mathrm r}}(\widehat{\Theta}).\nonumber
\end{equation}
For any \(\Theta,\Theta'\in\Omega_{\mathrm r}\) and
\(Z'\in\partial f_{\mathrm{Sh},r}(\Theta')\), the Bregman
divergence reduces to the sum of row-wise KL divergences,
\[
D_{f_{\mathrm{Sh},r}}(\Theta,\Theta';Z')
=
\sum_{i=1}^{m}
D_{\mathrm{KL}}(\Theta_{i,:},\Theta'_{i,:})
\ge
\frac{1}{2m}\|\Theta-\Theta'\|_1^2,
\]
where the inequality follows from Pinsker's and Cauchy's inequalities.
Thus, \(f_{\mathrm{Sh},r}\) is \(1/m\)-strongly convex on
\(\Omega_{\mathrm r}\) with respect to \(\|\cdot\|_1\). The update becomes the row-wise normalized multiplicative rule
\begin{tcolorbox}[ colback=black!2, colframe=black, boxrule=0.6pt, arc=0pt, left=6pt, right=6pt, top=4pt, bottom=4pt ]
\begin{equation} \label{update.row.entropic}
\begin{split}
\widehat{\Theta}_{t+1}
&=
\mathcal{N}_{\mathrm r}
\left(
\widehat{\Theta}_{t}\odot
\exp_{\mathrm{el}}\left(-\eta_{t+1}\nabla J_t(\widehat{\Theta}_t)\right)
\right),\\
\eta_{t+1} &=\frac{2}{m} \frac{\,J_t(\widehat{\Theta}_t)} {\epsilon_{t+1} +\|\nabla J_t(\widehat{\Theta}_t)\|_\infty^2},
\end{split}
\end{equation}    
\end{tcolorbox}
\noindent where \(\mathcal{N}_{\mathrm r}\) denotes row-wise normalization. If
\(\widehat{\Theta}_0=\mathbf{1}_{m\times k}/k\), then, for any
\(Z_0\in\partial f_{\mathrm{Sh},r}(\widehat{\Theta}_0)\),
\begin{equation*}
D_{f_{\mathrm{Sh},r}}(\Theta,\widehat{\Theta}_0;Z_0)
=
\sum_{i=1}^{m}\sum_{j=1}^{k}\Theta_{ij}\ln(k\Theta_{ij})
\le m\ln (k).
\end{equation*}
Theorem~\ref{Theorem.regret}, with \(\mu=1/m\) and
\(\|\cdot\|_*=\|\cdot\|_\infty\), gives
\begin{equation}\label{eq:row.stochastic.regret.bound}
\operatorname{Reg}(T)
\le
\sqrt{
\frac{m^2\ln (k)}{2}
\sum_{t=0}^{T-1}
\left(
\epsilon_{t+1}+
\| \nabla J_t(\widehat{\Theta}_t)\|_\infty^2
\right)
}.\nonumber
\end{equation}
\end{remark}

\subsection{Spectraplex and von Neumann Entropy}
We consider the case where the unknown matrix $\Theta$ belongs to the
spectraplex (the set of density matrices), i.e.,
\begin{equation}\label{eq:Spectrahedron}
\Omega_{\mathrm s}
:=
\left\{
\widehat{\Theta} \in \mathbb{S}_+^{k}
\;:\;
\Tr(\widehat{\Theta})=1
\right\}.
\end{equation}
In this subsection, the parameter matrix is square, i.e., $m=k$.
In this setting, we equip $\mathbb{S}^{k}$ with the nuclear norm
$\|\cdot\|_{\mathrm{S}_1}$, whose dual norm is the operator norm $\|\cdot\|_{\mathrm{S}_\infty}$.
On $\Omega_{\mathrm s}$, we employ the update law \eqref{eq:update.law}
with the von Neumann entropy mirror map
\begin{equation}\label{eq:vn.potential}
f_{\mathrm{vN}}(\widehat{\Theta})
:=
\Tr\bigl(\widehat{\Theta} \ln (\widehat{\Theta}) - \widehat{\Theta}\bigr)
+
\iota_{\Omega_{\mathrm s}}(\widehat{\Theta}),
\end{equation}
where $\iota_{\Omega_{\mathrm s}}$ denotes the indicator function of $\Omega_{\mathrm s}$.
By the quantum Pinsker inequality, \(f_{\mathrm{vN}}\) is \(1\)-strongly convex on
\(\Omega_{\mathrm s}\) with respect to \(\|\cdot\|_{\mathrm{S}_1}\). Moreover, for any $\Theta$ in the relative interior
\(
\operatorname{ri}(\Omega)
=
\left\{
\Theta \succ 0
\;:\;
\Tr(\Theta)=1
\right\},
\)
its subdifferential admits the characterization
\(\partial f_{\mathrm{vN}}(\Theta)
=
\left\{
\ln (\Theta) + \lambda I
:\;
\lambda \in \mathbb{R}
\right\},\)
where $I$ denotes the identity matrix.
The conjugate of $f_{\mathrm{vN}}$ is given by
\begin{equation*}
f_{\mathrm{vN}}^*(Z)
=
\ln\!\bigl(\Tr(\exp(Z))\bigr)+1,
\end{equation*}
which is everywhere differentiable, with gradient
\begin{equation*}
\nabla f_{\mathrm{vN}}^*(Z)
=
\frac{\exp(Z)}{\Tr(\exp(Z))},
\end{equation*}
where $\exp(\cdot)$ denotes the matrix exponential.

To preserve the symmetry of the parameter estimates, we use the symmetric part of the gradient, defined by
\begin{equation}\label{eq:sym.gradient.definition}
\operatorname{Sym}\left( \nabla J_t(\widehat{\Theta}_t) \right) := \frac{1}{2} \left( \nabla J_t(\widehat{\Theta}_t) + \nabla J_t(\widehat{\Theta}_t)^\top \right) \nonumber.
\end{equation}
Indeed, for every $H\in\mathbb{S}^k$, we have
\small
\begin{align*}
\left\langle \operatorname{Sym}\bigl(\nabla J_t(\widehat{\Theta}_t)\bigr),H\right\rangle 
= \frac{1}{2}\left\langle \nabla J_t(\widehat{\Theta}_t),H\right\rangle + \frac{1}{2}\left\langle \nabla J_t(\widehat{\Theta}_t)^\top,H\right\rangle &= \frac{1}{2}\left\langle \nabla J_t(\widehat{\Theta}_t),H\right\rangle + \frac{1}{2}\left\langle \nabla J_t(\widehat{\Theta}_t),H^\top\right\rangle \\
&= \left\langle \nabla J_t(\widehat{\Theta}_t),H\right\rangle.
\end{align*}
\normalsize
Thus, when $\Theta,\widehat{\Theta}_t\in\mathbb{S}^k$, replacing the gradient by its symmetric part does not change the key inner-product identity \eqref{eq:key.inner.product} used in the stability and regret analyses of Section~\ref{sec:main.result}.

Applying the update law \eqref{eq:update.law.main}, any choice of
$Z_t\in\partial f_{\mathrm{vN}}(\widehat{\Theta}_t)$ can be written as
\(
Z_t
=
\ln \widehat{\Theta}_t + \lambda_t I
\)
for some $\lambda_t\in\mathbb{R}$.
The dual update
\(
Z_{t+1}
=
Z_t - \eta_{t+1}  \operatorname{Sym}(\nabla J_t(\widehat{\Theta}_t))
\)
thus yields
\begin{equation}
Z_{t+1}
=
\ln (\widehat{\Theta}_t)
-
\eta_{t+1}  \operatorname{Sym}(\nabla J_t(\widehat{\Theta}_t))
+
\lambda_t I.\nonumber
\end{equation}
Applying the primal recovery $\widehat{\Theta}_{t+1}=\nabla f_{\mathrm{vN}}^*(Z_{t+1})$ gives
\begin{tcolorbox}[ colback=black!2, colframe=black, boxrule=0.6pt, arc=0pt, left=6pt, right=6pt, top=4pt, bottom=4pt ]
\begin{equation}\label{eq:update.von}
\begin{split}
\widehat{\Theta}_{t+1}
&=
\frac{
\exp\bigl(\ln (\widehat{\Theta}_t) - \eta_{t+1}  \operatorname{Sym}(\nabla J_t(\widehat{\Theta}_t))\bigr)
}{
\Tr\left(
\exp\bigl(\ln (\widehat{\Theta}_t) - \eta_{t+1}  \operatorname{Sym}(\nabla J_t(\widehat{\Theta}_t))\bigr)
\right)
},\\
\eta_{t+1} &= \frac{2\,J_t(\widehat{\Theta}_t)} {\epsilon_{t+1} +\|\operatorname{Sym}(\nabla J_t(\widehat{\Theta}_t))\|_{\mathrm{S}_\infty}^2}.
\end{split}
\end{equation}
\end{tcolorbox}
The following corollary provides the resulting regret bound.
\begin{corollary}[Regret bound over the spectraplex]
Consider the closed-loop system \eqref{eq.dynamics} under the control law
\eqref{eq:control.law} and the update law \eqref{eq:update.von}. Suppose
that \(\Theta\in\Omega_{\mathrm s}\) and
\(\widehat{\Theta}_0=I/k\). Then
\begin{equation}\label{eq:corollary.spectraplex}
\operatorname{Reg}(T)
\le
\sqrt{
\frac{\ln (k)}{2}
\sum_{t=0}^{T-1}
\left(
\epsilon_{t+1}+
\|\operatorname{Sym}(\nabla J_t(\widehat{\Theta}_t))\|_{\mathrm{S}_\infty}^2
\right)
}.
\end{equation}
\end{corollary}

\begin{proof}
Let \(Z_0\in\partial f_{\mathrm{vN}}(\widehat{\Theta}_0)\). Since
\(\widehat{\Theta}_0=I/k\), we have
\(Z_0=\ln(\widehat{\Theta}_0)+\lambda_0 I=(-\ln (k)+\lambda_0)I\) for some
\(\lambda_0\in\mathbb R\). Hence, since
\(\Theta,\widehat{\Theta}_0\in\Omega_{\mathrm s}\), we have
\(\langle \Theta-\widehat{\Theta}_0,Z_0\rangle=0\). Therefore,
\begin{align}
D_{f_{\mathrm{vN}}}(\Theta,\widehat{\Theta}_0;Z_0)
=
f_{\mathrm{vN}}(\Theta)-f_{\mathrm{vN}}(\widehat{\Theta}_0)
-\langle \Theta-\widehat{\Theta}_0,Z_0\rangle =
\Tr(\Theta\ln(\Theta))+\ln (k) 
\le \ln (k), \label{eq:vn.initial.divergence}
\end{align}
where the last inequality follows from
\(\Tr(\Theta\ln(\Theta))=\sum_i\alpha_i\ln(\alpha_i)\le0\), with
\(\{\alpha_i\}\) denoting the positive eigenvalues of \(\Theta\).
Applying \eqref{eq:regret.bound.theorem3} with \(\mu=1\),
\(\|\cdot\|_*=\|\cdot\|_{\mathrm{S}_\infty}\), replacing the $\nabla J_t(\widehat{\Theta}_t)$ with $\operatorname{Sym}(\nabla J_t(\widehat{\Theta}_t))$ and  using
\eqref{eq:vn.initial.divergence} yields \eqref{eq:corollary.spectraplex}.
\end{proof}
\section{Extension to Stochastic Systems} \label{sec:stochastic}
In this section, we briefly discuss the effect of process noise on the proposed framework. Consider the stochastic system
\begin{equation}\label{eq.dynamics.noise}
    X_{t+1} = \mathcal{A}(X_t) + \mathcal{B}\big(\Theta \Psi(X_t,t) + U_t\big)+W_t,
\end{equation}
where \(W_t\in\mathbb{R}^{n_1\times n_2}\) is an independent, zero-mean white-noise process, independent of the initial state and parameter estimate. Under the
control law \eqref{eq:control.law}, the residual becomes
\small
\begin{align*}
  \widetilde X_{t+1}:=X_{t+1}
    -
    (\mathcal{A}-\mathcal{B}\mathcal{K})(X_t)-
    \mathcal{B}(U^{\rm d}_t) 
=
\mathcal B\big((\Theta-\widehat\Theta_t)\Psi(X_t,t)\big)+W_t .
\end{align*}
\normalsize
For each realization of \(X_t\), define the instantaneous loss
\begin{equation*}
 J_t(\widehat\Theta;X_t)
:=
\frac12
\left\|
\mathcal B\big((\Theta-\widehat{\Theta})\Psi(X_t,t)\big)
\right\|_{\mathrm{F}}^2,
\end{equation*}
Conditional on \(X_t\), this is a deterministic function of \(\widehat\Theta\). The corresponding expected regret is defined as
\begin{equation}\label{eq:regret.noise}
\mathrm{Reg}(T)
:=
\sum_{t=0}^{T-1} \mathbb{E} J_t(\widehat\Theta_t;X_t),
\end{equation}
where the expectation is taken over the randomness induced by the process noise. Even in the noisy setting, a controller of the form \eqref{eq:control.law} with full knowledge of the true parameter \(\Theta\), namely \(U_t(X_t,\Theta)\), incurs zero regret because \( J_t(\Theta;X_t)=0 \) for every realization of \(X_t\).

After observing \(X_{t+1}\), define the noisy gradient estimator
\begin{equation*}
\widehat G_t
:=
-\mathcal B^\top(\widetilde X_{t+1})\Psi(X_t,t)^\top.
\end{equation*}
Using the conditional zero-mean property of \(W_t\), we obtain
\begin{align*}
    \mathbb{E}\!\left[ \widehat G_t\mid X_t,\widehat{\Theta}_t\right]
    = -\mathcal{B}^\top
    \mathbb{E}\!\left[\widetilde{X}_{t+1}\mid X_t,\widehat{\Theta}_t\right]
    \Psi(X_t,t)^\top &= -\mathcal{B}^\top
    \mathcal{B}\big((\Theta-\widehat{\Theta}_t)\Psi(X_t,t)\big)
    \Psi(X_t,t)^\top\\&=\nabla J_t(\widehat{\Theta}_t;X_t).
\end{align*}
Thus, the update law \eqref{eq:update.law}, when driven by
\(\widehat G_t\), can be viewed as a stochastic online mirror-descent
scheme with an unbiased conditional gradient oracle. The corresponding
Polyak-type step size takes the form
\begin{equation}\label{eq:stochastic.stepsize}
\eta_{t+1}
=
\frac{\mu \|\widetilde{X}_{t+1}\|^2_{\mathrm{F}}}
{\epsilon_{t+1}+\|\widehat G_t\|_*^2}.
\end{equation}

The preceding discussion motivates the following question: To what extent can the stability and regret guarantees of Section~\ref{sec:main.result} be recovered in the presence of process noise?

\textbf{Empirical observations.}
In Section~\ref{sec:simu}, we empirically demonstrate that the proposed method
remains stable and continues to exhibit essentially dimension-independent
performance in the presence of process noise, while significantly outperforming
Euclidean-based adaptive schemes. In contrast, recursive least-squares-based
adaptive controllers may become ill-conditioned and exhibit unstable behavior
in high-dimensional noisy settings.

\textbf{Robust modifications.}
Standard robustification mechanisms, such as dead-zone modifications or
projection-based updates, can also be incorporated to guarantee boundedness of
the parameter estimates \(\widehat{\Theta}_t\). For instance, the constrained
mirror maps used in Section~\ref{sec:simplex} provide a natural way to enforce
such boundedness through the inclusion of a set-indicator function, which keeps
the iterates within a prescribed constraint set.

\textbf{Theoretical challenges.}
A complete theoretical treatment, however, requires overcoming two key
obstacles. First, the stochastic closed-loop system must be shown to satisfy
appropriate stability or moment bounds without imposing uniformly bounded
regressors. This is important because the deterministic analysis relies only on
the linear-growth condition \eqref{eq:linear-growth-regressor}, which is more
natural for dynamical systems but more delicate under process noise. Second,
the Polyak-type step size \eqref{eq:stochastic.stepsize} introduces additional
technical difficulties. Indeed, in the noisy setting, the step size depends on
the random residual \(\widetilde{X}_{t+1}\), and therefore explicitly depends
on the disturbance \(W_t\). Consequently, the step size and the gradient
estimator become statistically coupled, preventing a direct application of
standard unbiased stochastic approximation arguments.
Therefore, extending the regret bound \eqref{eq:regret.bound.theorem3} to the
stochastic setting likely requires either an alternative step-size mechanism,
such as predictable or delayed step sizes, or a more refined stochastic
analysis. We leave such developments for future work. Nevertheless, the unbiased-gradient structure derived above, together with the empirical evidence in Section~\ref{sec:simu}, suggests that the canonical
\(O(\sqrt{T})\) stochastic regret rate, with logarithmic dependence on the
ambient dimension for sparse and simplex-type parameter structures, should
remain achievable under suitable stability and moment conditions.

\section{Simulations: Online optimization with adaptive control}\label{sec:simu}

\begin{figure}
    \centering
    \subfloat[Noiseless setting, \(W_t=0\)]{
        \includegraphics[width=0.49\linewidth]{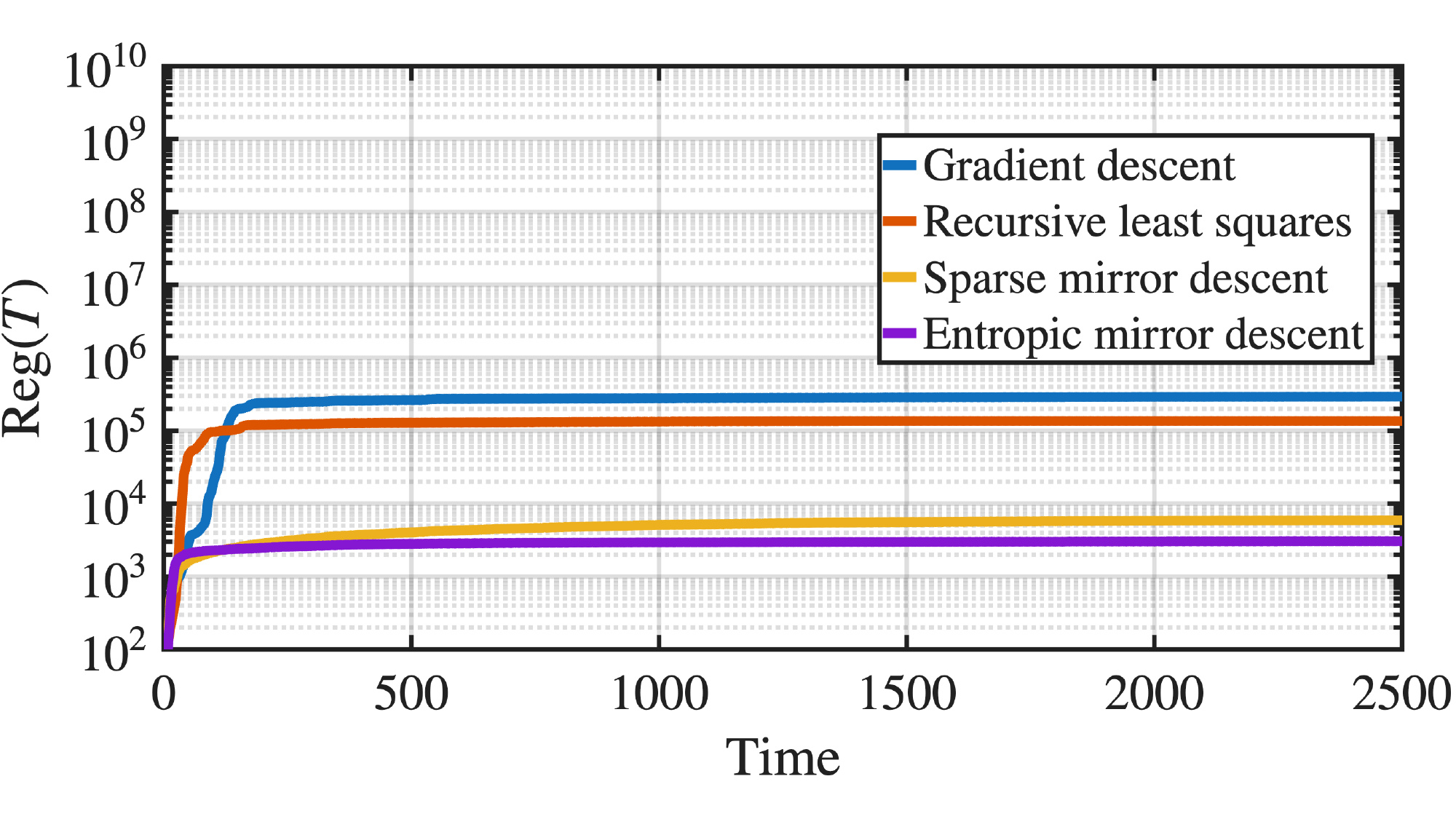}
    }
    \hspace{-0.4cm}
    \subfloat[Noisy setting, \(W_t\neq 0\)]{
        \includegraphics[width=0.49\linewidth]{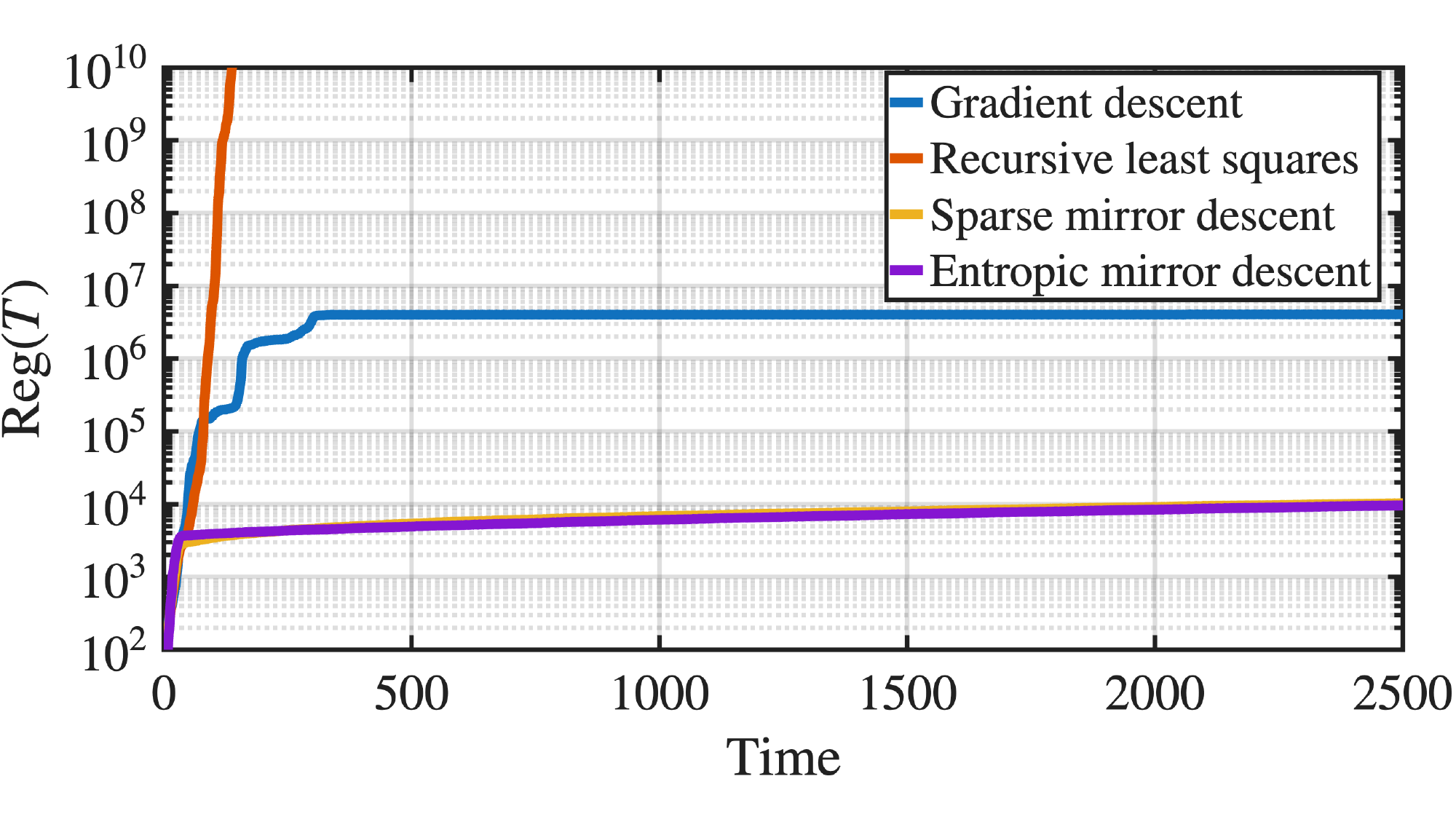}
    }
    \caption{Evolution of the regret \(\operatorname{Reg}(T)\) for the system \eqref{eq:sim.sys} under the control law \eqref{eq:control.sim} with four different parameter-update algorithms in the noiseless and noisy settings.}
    \label{fig:regret}
\end{figure}

\begin{figure*}[t]
    \centering
    \hspace{-0.4cm}
    \subfloat[Gradient descent]{
        \includegraphics[width=0.49\textwidth]{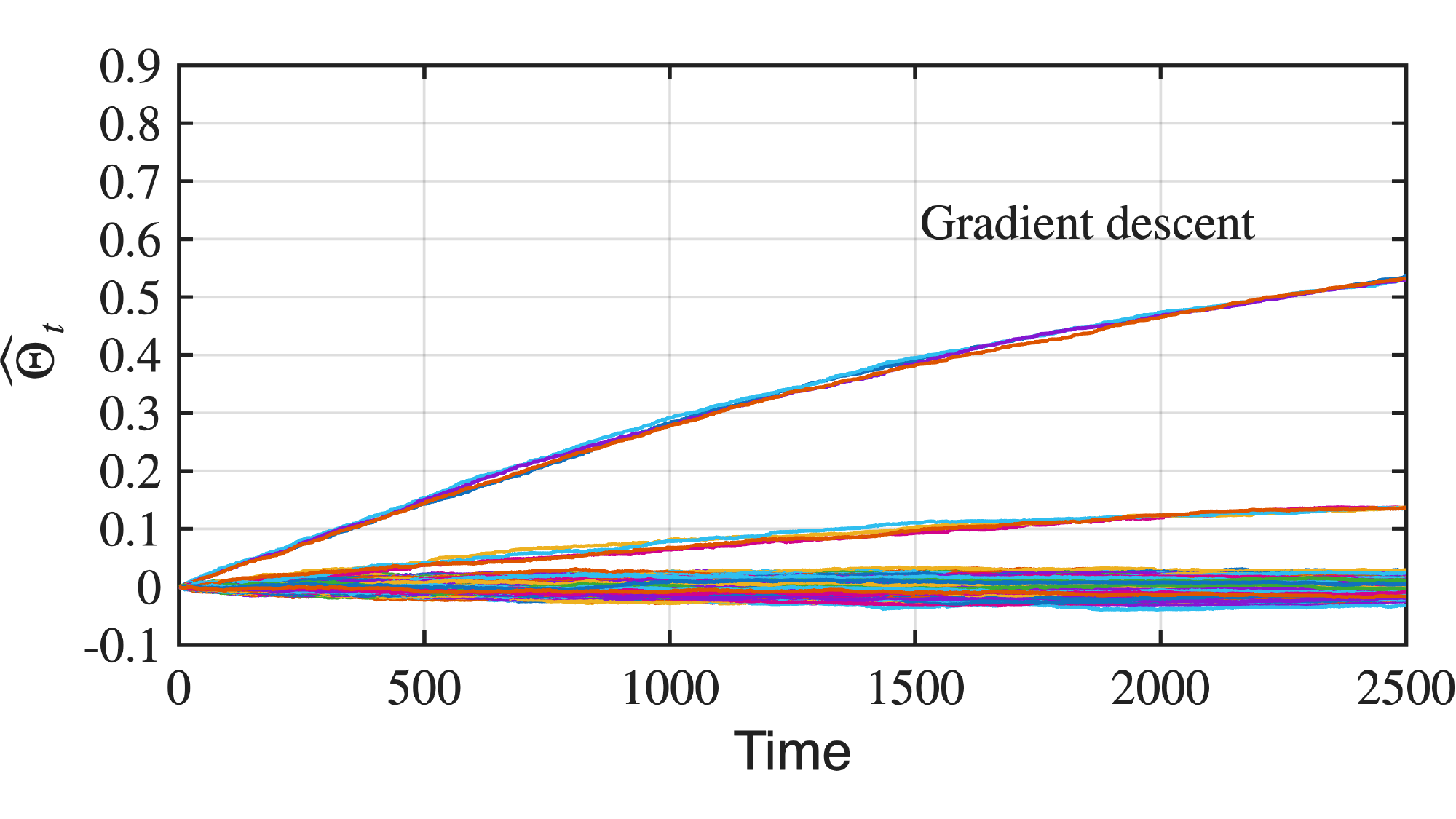}
        \label{fig:par_without_noise_gd}
    } \hspace{-0.4cm}
    \subfloat[Recursive least squares]{
        \includegraphics[width=0.49\textwidth]{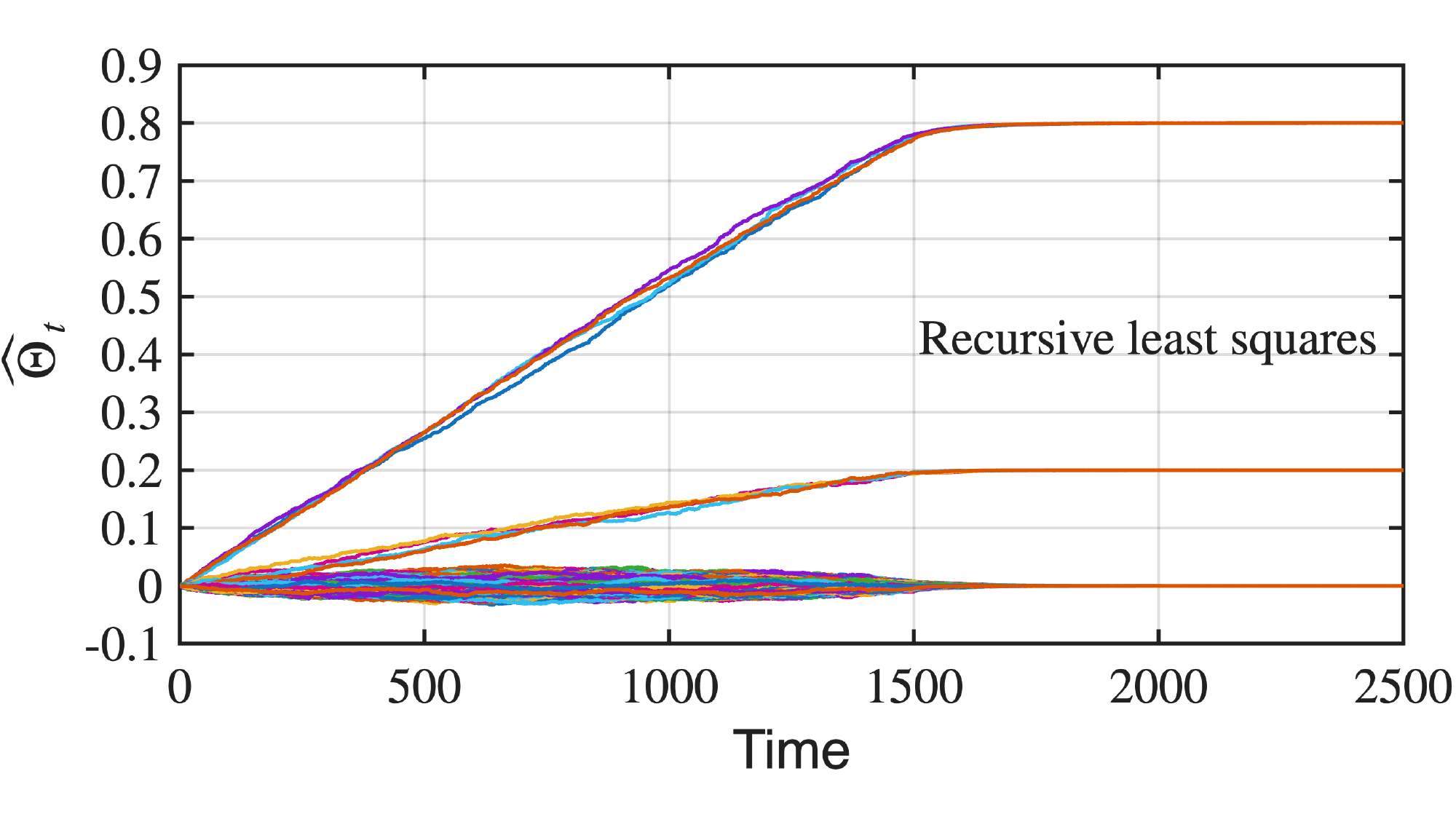}
        \label{fig:par_without_noise_rls}
    } \hspace{-0.4cm}
    \subfloat[Sparse mirror descent]{
        \includegraphics[width=0.49\textwidth]{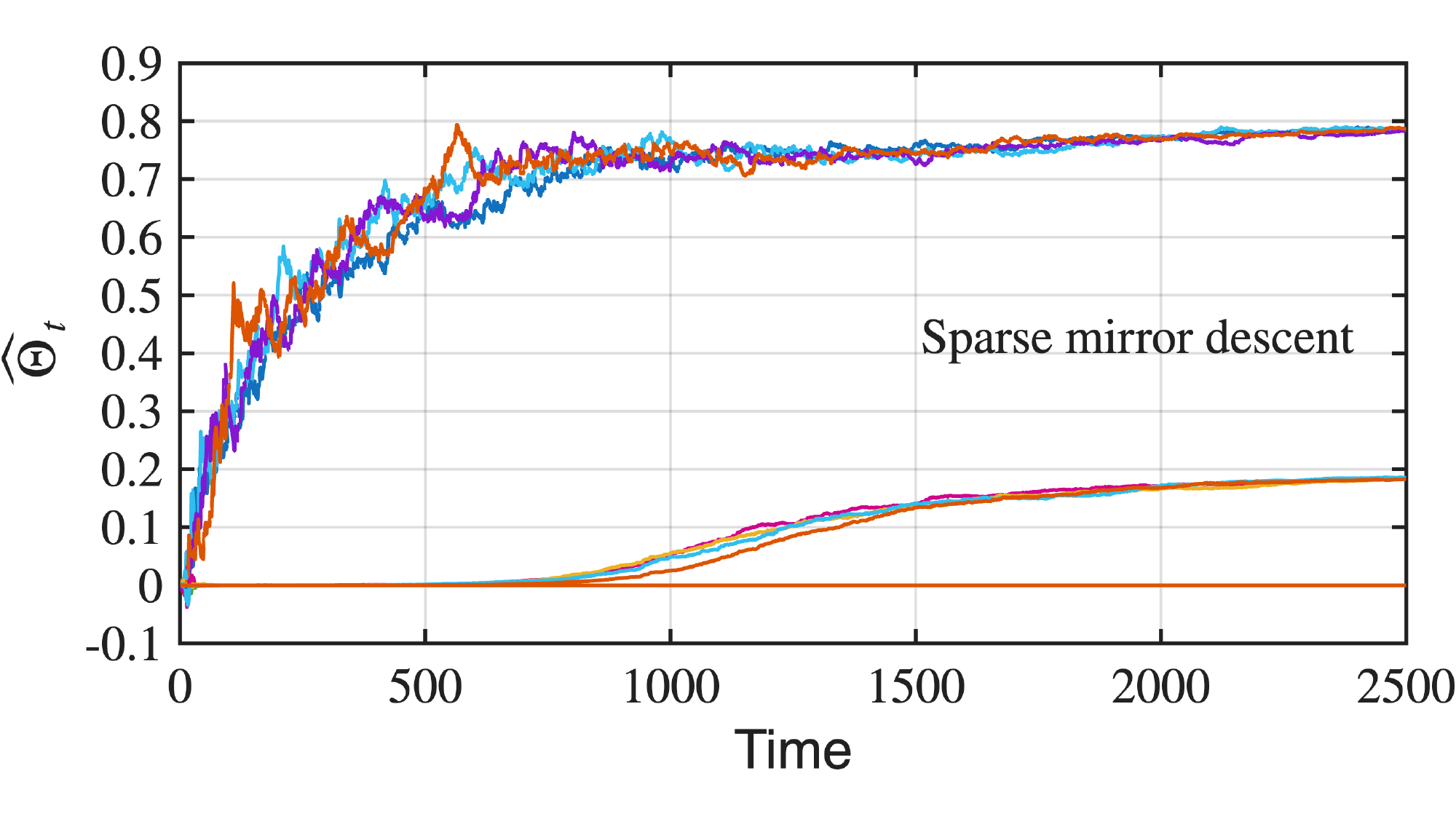}
        \label{fig:par_without_noise_sparse}
    } \hspace{-0.4cm}
    \subfloat[Entropic mirror descent]{
        \includegraphics[width=0.48\textwidth]{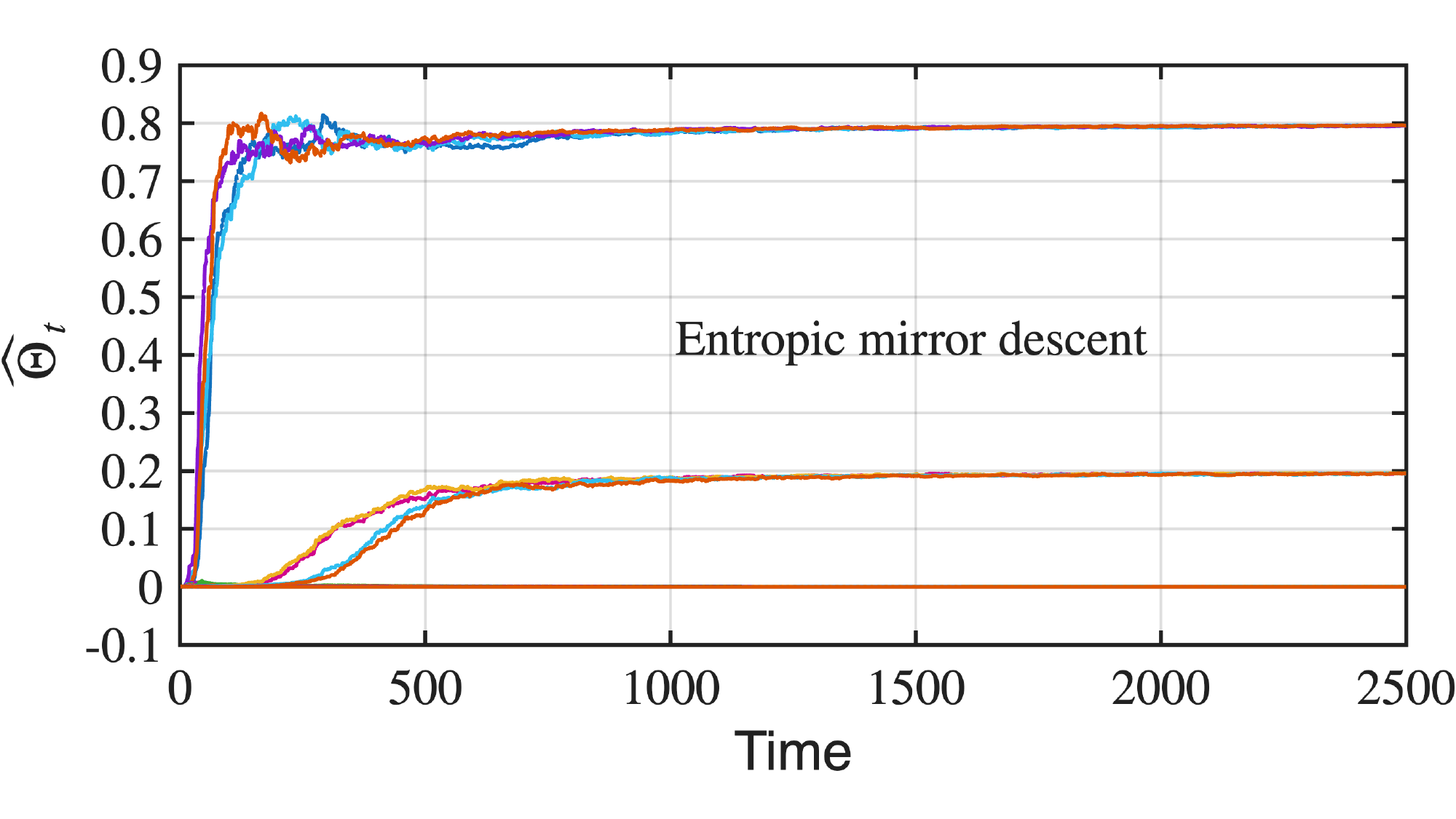}
        \label{fig:par_without_noise_entropic}
    }
    \caption{Parameter estimates in the noiseless setting, \(W_t=0\).}
    \label{fig:par_all_without_noise}
\end{figure*}

Consider a multi-agent system consisting of $N$ agents, whose collective dynamics are governed by
\begin{equation}\label{eq:sim.sys}
    X_{t+1} = \Theta \Psi(X_t, t) + U_t+W_t,
\end{equation}
where \(X_t\in\mathbb{R}^{n\times N}\) denotes the collective state at time \(t\), \(U_t\in\mathbb{R}^{n\times N}\) is the control input, and \(W_t\in\mathbb{R}^{n\times N}\) is a stochastic disturbance. The matrix $\Theta \in \mathbb{R}^{n \times k}$ represents unknown system parameters, while the regressor $\Psi(X_t, t)$ accounts for the inherent internal dynamics, potential inter-agent coupling, and exogenous disturbances. Furthermore, let $\{\mathcal{F}_t: \mathbb{R}^{n \times N} \rightarrow \mathbb{R}_{+}\}_{t \geq 0}$ be a sequence of time-varying, $\mathcal{L}$-smooth, and convex cost functions that characterize the performance objective at each time step.

The control objective is to design an adaptive controller such that the state of the system \eqref{eq:sim.sys} tracks the reference trajectory $\{X^{\mathrm d}_t\}_{t \in \mathbb{Z}_{+}}$ produced by the online gradient descent iteration
\begin{equation*}
    X_{t+1}^{\mathrm d} = X_t^{\mathrm d} - \frac{1}{\mathcal{L}}\nabla \mathcal{F}_t(X_t^{\mathrm d}).
\end{equation*}
For the numerical study, we consider the Laplacian-coupled quadratic objective
\small
\begin{equation}\label{eq:sim.cost}
    \mathcal{F}_t(X)
    =
    \frac{1}{2}\operatorname{Tr}\!\big((X-C_t)^\top H (X-C_t)\big)
    +
    \frac{1}{2}\operatorname{Tr}(X L_g X^\top),
\end{equation}
\normalsize
where $H\in\mathbb{S}_+^n$ is a symmetric positive semidefinite matrix, possibly singular, $C_t\in\mathbb{R}^{n\times N}$ is a time-varying reference matrix, and $L_g\in\mathbb{R}^{N\times N}$ is the graph Laplacian of the communication network. The second term in \eqref{eq:sim.cost} is the standard Laplacian regularizer
\begin{equation*}
       \operatorname{Tr}(X L_g X^\top)
    =
    \frac{1}{2}\sum_{i,j=1}^N w_{ij}\|x_i-x_j\|_2^2, 
\end{equation*}
which penalizes disagreement between neighboring agents. The gradient of \eqref{eq:sim.cost} is given by
\begin{equation*}
  \nabla \mathcal{F}_t(X)=H(X-C_t)+ X L_g.  
\end{equation*}  
Thus, the gradient is Lipschitz continuous with constant 
\begin{equation*}
 \mathcal{L} = \lambda_{\max}(H)+\lambda_{\max}(L_g).   
\end{equation*}
 In the simulations, $L_g$ is chosen as the Laplacian of an unweighted ring graph,
namely,
\begin{equation*}
(L_g)_{ij}=
\begin{cases}
2, & i=j,\\
-1, & j=i\pm 1 \pmod N,\\
0, & \text{otherwise},
\end{cases}
\end{equation*}
so that each agent communicates only with its two immediate neighbors. Moreover,
we set
\begin{equation*}
    H=\operatorname{diag}(h_1,\dots,h_n)\in \mathbb{S}_+^n,\quad
    h_i=
\begin{cases}
1, & i \text{ odd},\\
0, & i \text{ even},
\end{cases}
\end{equation*}
so that the tracking term penalizes only the odd-indexed state coordinates. Following the control law \eqref{eq:control.law}, we set 
\begin{equation}\label{eq:control.sim}
   U_t = -\widehat{\Theta}_t\Psi(X_t,t) + X_t^{\mathrm d} - \frac{1}{\mathcal L} \nabla\mathcal F_t(X_t^{\mathrm d}).
\end{equation}
In the simulations, we take \(n=N=4\) and \(k=3000\). The reference signal is chosen as
\[
    (C_t)_{ij}
    =
    (-1)^{i+j}\sin(0.05t),
    \quad i=1,\dots,n,\quad j=1,\dots,N.
\]
The feature map is set as
\begin{equation*}
    \Psi(X_t,t)
    =
    \Gamma\big(X_t\odot \sin(X_t)+\mathbf{1}_{n\times N}\big),
\end{equation*}
where \(\Gamma\in\mathbb{R}^{k\times n}\) is a random feature matrix with independent Rademacher entries, i.e.,
\begin{equation*}
    \mathbb{P}(\Gamma_{ij}=1)=\mathbb{P}(\Gamma_{ij}=-1)=\frac 12.
\end{equation*}
The true parameter matrix \(\Theta\) is chosen to be sparse and row-stochastic. Specifically, for each row \(i=1,\dots,n\), we set
\begin{equation*}
    \Theta_{ij}
    =
    0.8\times\mathbf{1}_{\{j=i\}}
    +
    0.2\times\mathbf{1}_{\{j=k-i\}},
    \qquad j=1,\dots,k.
\end{equation*}
where \(\mathbf{1}_{\{\mathcal E\}}\) equals \(1\) if \(\mathcal E\) holds and \(0\) otherwise. Thus, each row of \(\Theta\) has two nonzero entries, whose values sum to one.  

\begin{figure*}[t]

    \hspace{-0.4cm}
    \subfloat[Gradient descent]{
        \includegraphics[width=0.48\textwidth]{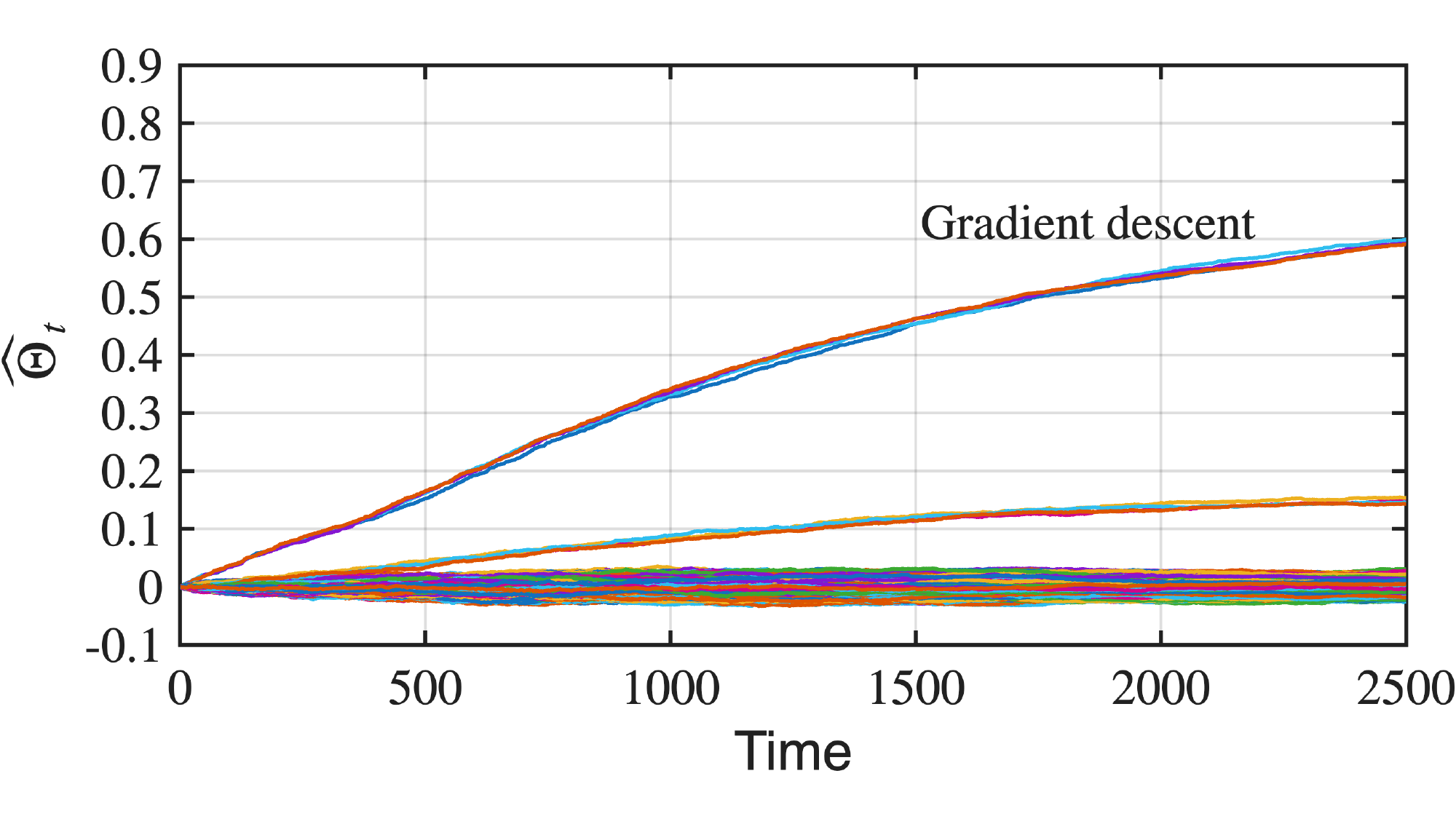}
        \label{fig:par_with_noise_gd}
    } \hspace{-0.4cm}
    \subfloat[Recursive least squares]{
    \includegraphics[width=0.48\textwidth]{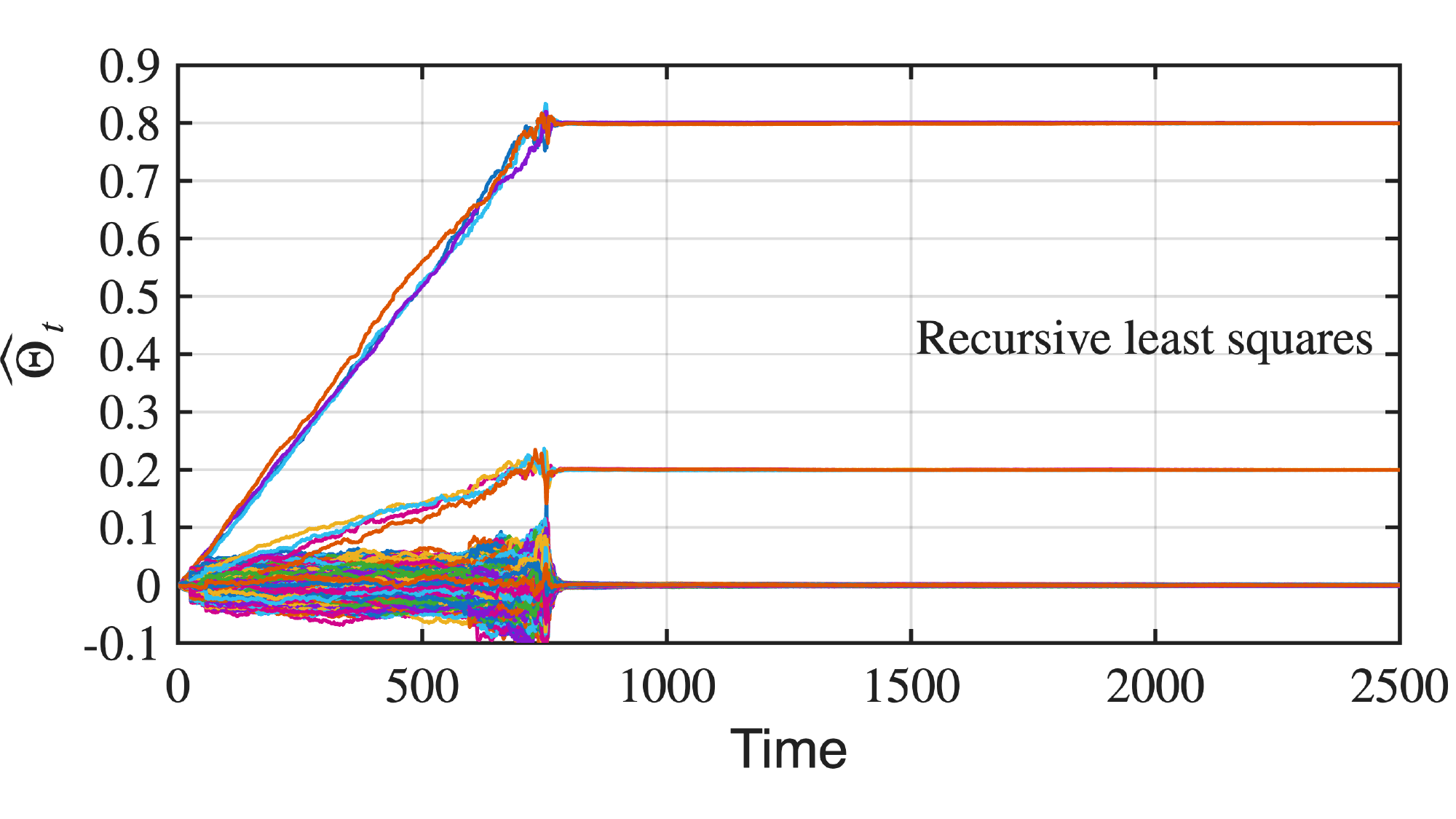}
        \label{fig:par_with_noise_rls}
    } \hspace{-0.4cm}
    \subfloat[Sparse mirror descent]{
        \includegraphics[width=0.48\textwidth]{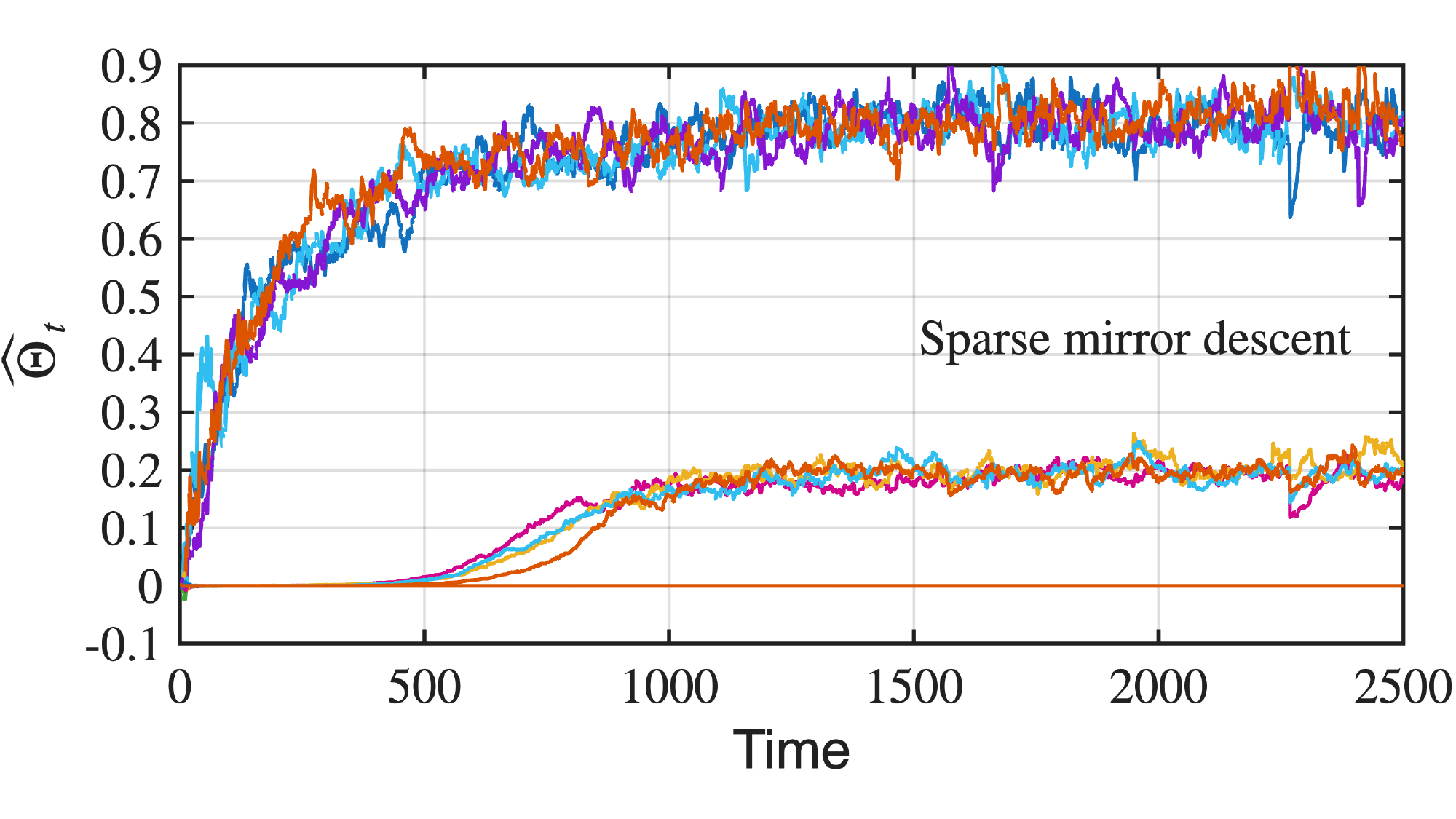}
        \label{fig:par_with_noise_sparse}
    } \hspace{-0.4cm}
    \subfloat[Entropic mirror descent]{
        \includegraphics[width=0.48\textwidth]{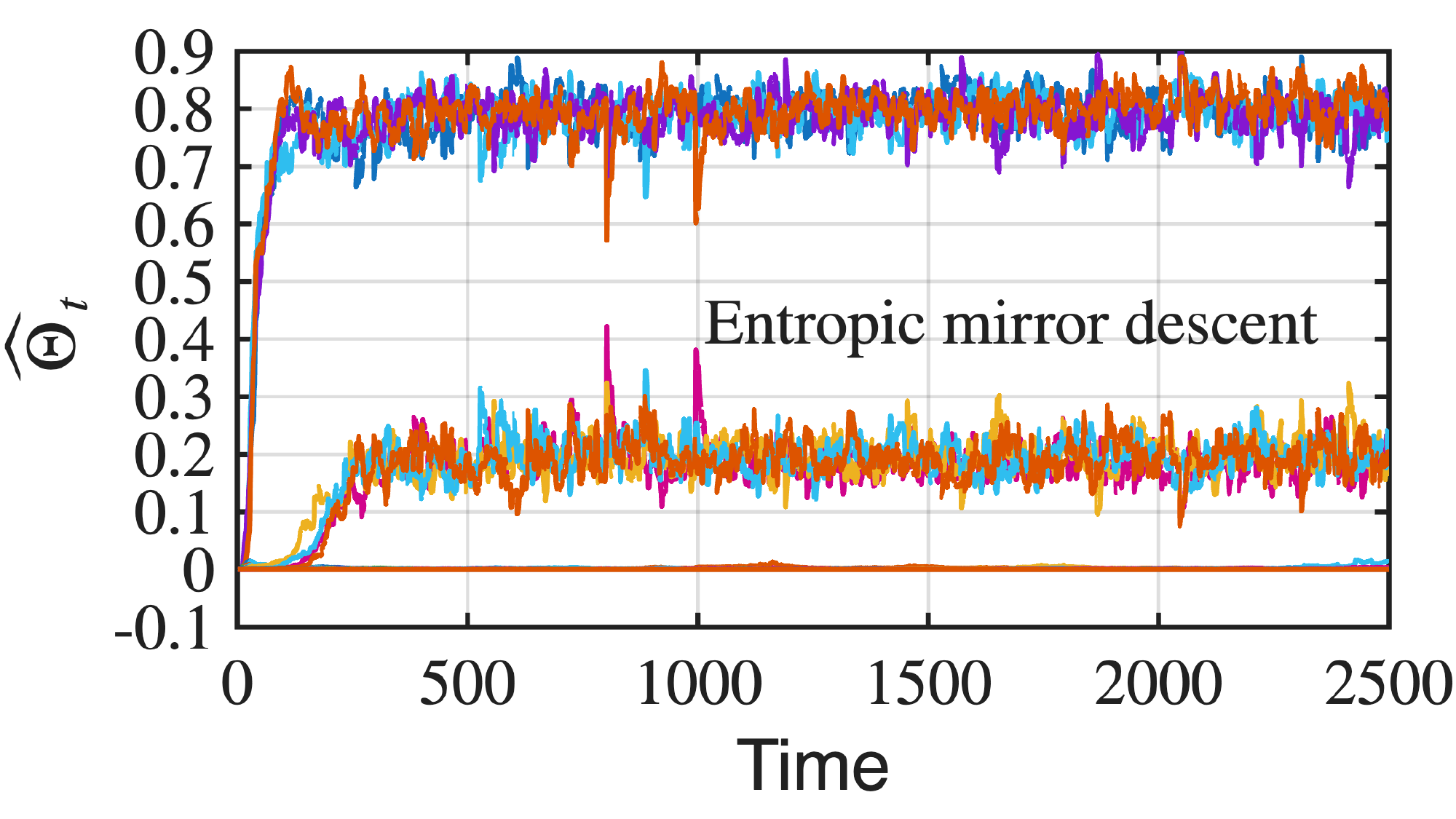}
        \label{ffig:par_with_noise_entropic}
    } 
    \caption{Parameter estimates in the noisy setting, \(W_t\neq 0\).}
    \label{fig:par_all_with_nose}
\end{figure*}
We consider two simulation scenarios: a noiseless setting with $W_t=0$, and a noisy setting where the entries of $W_t$ are generated independently from the uniform distribution on $[-1,1]$. In both cases, we compare four parameter-estimation schemes: recursive least squares (RLS); gradient descent, corresponding to \eqref{eq:update.law} with the mirror map $f(\widehat{\Theta})=\frac{1}{2}\Vert{}\widehat{\Theta}\Vert{}_{\mathrm F}^2$; the row-wise entropic mirror descent update \eqref{update.row.entropic}; and the sparse mirror descent update \eqref{eq:sparse.update}.

Figures~\ref{fig:regret}(a) and \ref{fig:regret}(b) illustrate the evolution of the regret $\operatorname{Reg}(T)$ in the noiseless and noisy settings, respectively. As observed, both the entropic and sparse mirror descent methods significantly outperform gradient descent and RLS. Furthermore, the results highlight varying degrees of robustness to process noise: while both entropic and sparse mirror descent maintain their performance, RLS degrades substantially. Specifically, although RLS outperforms gradient descent in the noiseless case, its regret grows rapidly in the presence of noise. This behavior suggests that, in high-dimensional regimes, first-order methods such as mirror descent can offer greater robustness than RLS. This disparity may be attributed to the fact that RLS relies on recursively updating an inverse covariance matrix, which is prone to ill-conditioning and heightened sensitivity to persistent noise in high dimensions.

Figures~\ref{fig:par_all_without_noise} and \ref{fig:par_all_with_nose} show the corresponding parameter estimates obtained by each algorithm in the noiseless and noisy settings, respectively. The results demonstrate that the entropic and sparse mirror descent updates are better able to identify the low-dimensional sparse structure of the true parameter matrix \(\Theta\). In contrast, gradient descent and RLS tend to spread the estimation error across many coordinates, producing dense parameter estimates even though the true parameter matrix is sparse. It is also worth noting that, for entropic and sparse mirror descent, the step-size denominator depends on the $\ell_\infty$-norm of the gradient, which is much smaller in high dimensions than the Frobenius norm utilized in gradient descent. Consequently, the resulting step size for gradient descent can be considerably smaller in high-dimensional regimes, causing its estimates to evolve at a much slower rate. 

\section{Conclusions and Future Directions}\label{sec:conclusions}

This paper studied adaptive control of discrete-time nonlinear systems with high-dimensional, matrix-valued parametric uncertainty. We showed that standard Euclidean adaptive schemes, such as recursive least squares and gradient-descent-based updates, can exhibit transient performance degradation as the ambient dimension increases, even when the true parameter has low intrinsic complexity. To address this limitation, we developed mirror-descent-based adaptive laws that exploit the geometry of structured parameter classes, including sparse, low-rank, simplex-constrained, and spectraplex-constrained uncertainty sets. The proposed framework guarantees boundedness of all closed-loop signals, asymptotic tracking, and regret bounds with at most logarithmic dependence on the ambient dimension.

Several directions remain open for future research. An important next step is to extend the regret bound in \eqref{eq:regret.bound.theorem3} to stochastic systems subject to process noise. Another direction is to sharpen the dependence of the regret bound on the horizon \(T\) under stronger excitation conditions, such as persistence of excitation \cite{WILLEMS2005325}. Finally, incorporating prediction-based online learning methods, including gradient prediction and optimistic mirror descent \cite{rakhlin,pedro}, may further improve transient performance in predictable or slowly varying environments.

\appendix
\section{Proof of auxiliary Lemmas}\label{sec:appendix}
\subsection{Proof of Lemma \ref{lemma.upper.bound} }
Throughout the proof, $c,c_1,$ and $c_2$ denote positive constants whose values may change from line to line. Since all norms are equivalent in finite-dimensional spaces, there exists a constant $c>0$ such that
\begin{equation*}
\| \nabla J_t(\widehat{\Theta}_t)\|_{*} \le c \| \nabla J_t(\widehat{\Theta}_t)\|_{\mathrm{F}}.
\end{equation*}
Using the definition of $ \nabla J_t(\widehat{\Theta}_t)$ and the submultiplicativity of the Frobenius norm, we find that
\begin{align}
\| \nabla J_t(\widehat{\Theta}_t)\|_{*}  \le c \| \nabla J_t(\widehat{\Theta}_t)\|_{\mathrm{F}}
= c \big\|\mathcal{B}^{\top}(\widetilde{X}_{t+1}) \Psi(X_t,t)^{\top}\big\|_{\mathrm{F}} &\le c \|\mathcal{B}^{\top}(\widetilde{X}_{t+1})\|_{\mathrm{F}} \, \|\Psi(X_t,t)\|_{\mathrm{F}} \nonumber\\&\le c \|\mathcal{B}^{\top}\|_{\mathrm{op}} \, \|\widetilde{X}_{t+1}\|_{\mathrm{F}} \, \|\Psi(X_t,t)\|_{\mathrm{F}}.
\label{eq:proof.lemma.bound.E}
\end{align}
where $\|\mathcal{B}^\top\|_{\mathrm{op}}$ denotes the operator norm of $\mathcal{B}^\top$.     Next, consider the recursion 
\begin{equation*}
    X_{t+1}= (\mathcal{A} - \mathcal{B}\mathcal{K}) (X_t)  +\widetilde{X}_{t+1}+ \mathcal{B}(U^{\rm d}_t).
\end{equation*}
Since the closed-loop operator $\mathcal{A}-\mathcal{B}\mathcal{K}$ is Schur stable, there exist constants $c>0$ and $\gamma\in(0,1)$ such that \(\|(\mathcal{A}-\mathcal{B}\mathcal{K})^t\|_{\mathrm{op}}\le c\gamma^t\) for all \(t\ge 0\). Iterating the recursion yields
\begin{align*}
    \|X_t\|_{\mathrm{F}}
    &\le c\gamma^t\|X_0\|_{\mathrm{F}}
    +c\sum_{i=0}^{t-1} \gamma^{\,t-1-i}(\|\widetilde{X}_{i+1}\|_{\mathrm{F}}+\|\mathcal{B}(U^{\rm d}_i)\|_{\mathrm{F}}) \nonumber\\
    &\le c\gamma^t\|X_0\|_{\mathrm{F}}
    +\frac{c}{1-\gamma}\max_{0\le i\le t-1}(\|\widetilde{X}_{i+1}\|_{\mathrm{F}}+\|\mathcal{B}(U^{\rm d}_i)\|_{\mathrm{F}}) \nonumber\\
    &\leq c_1 +c_2  \max_{0\le i\le t-1}\|\widetilde{X}_{i+1}\|_{\mathrm{F}},
\end{align*}
 where in the second inequality we used the fact that the desired input $U^{\rm d}_t$ is bounded. Furthermore, due to the linear growth condition on $\Psi(X_t,t)$, we have
\begin{align*}
\|\Psi(X_t,t)\|_{\mathrm{F}} &\leq c_1+c_2 \|X_t\|_{\mathrm{F}}\leq c_1 +c_2  \max_{0\le i\le t-1}\|\widetilde{X}_{i+1}\|_{\mathrm{F}}.
\end{align*}
As a result, we find that
\begin{equation*}
    \|\widetilde{X}_{t+1}\|_{\mathrm{F}}\, \|\Psi(X_t,t)\|_{\mathrm{F}} \leq c_1 + c_2 \max_{0\le i\le t}\|\widetilde{X}_{i+1}\|_{\mathrm{F}}^2.
\end{equation*}
Finally, combining this bound with \eqref{eq:proof.lemma.bound.E}, we obtain
\begin{align*}
    \| \nabla J_t(\widehat{\Theta}_t)\|_{*}^2 \leq c \|\mathcal{B}^\top\|^2_{\mathrm{op}}\,\|\widetilde{X}_{t+1}\|^2_{\mathrm{F}}\,\|\Psi(X_t,t)\|^2_{\mathrm{F}}  \leq c_1 +c_2\max_{0\le i\le t}\|\widetilde{X}_{i+1}\|_{\mathrm{F}}^4,
\end{align*}
as claimed.

{
\printbibliography
}

\end{document}